\documentclass{amsart}
\usepackage{amsmath,amssymb,amsthm, amscd}
\usepackage{pifont,mathrsfs}
\usepackage{array,lastpage}
\usepackage{enumerate,xspace,pifont,shadow}
\usepackage{mathrsfs}
\usepackage [english]{babel}
\usepackage [autostyle, english = american]{csquotes}
\usepackage{xcolor}
\MakeOuterQuote{"}

\DeclareSymbolFont{AMSb}{U}{msb}{m}{n}
\DeclareMathSymbol{\Z}{\mathbin}{AMSb}{"5A}
\DeclareMathSymbol{\R}{\mathbin}{AMSb}{"52}
\DeclareMathSymbol{\N}{\mathbin}{AMSb}{"4E}
\DeclareMathSymbol{\Q}{\mathbin}{AMSb}{"51}

\newcommand{\qftp}{\textup{qftp}}

\newcommand{\dcl}{\textup{dcl}}
\newcommand{\acl}{\textup{acl}}
\newcommand{\Th}{\textup{Th}}

\newcommand{\mc}[1]{\mathcal{#1}}
\newcommand{\mf}[1]{\mathfrak{#1}}
\newcommand{\ob}[1]{\overline{#1}}

\def\Ind{\setbox0=\hbox{$x$}\kern\wd0\hbox to 0pt{\hss$\mid$\hss}
\lower.9\ht0\hbox to 0pt{\hss$\smile$\hss}\kern\wd0}

\def\Notind{\setbox0=\hbox{$x$}\kern\wd0\hbox to 0pt{\mathchardef
\nn=12854\hss$\nn$\kern1.4\wd0\hss}\hbox to
0pt{\hss$\mid$\hss}\lower.9\ht0 \hbox to
0pt{\hss$\smile$\hss}\kern\wd0}

\newtheorem{thm}{Theorem}[section]
\newtheorem{lem}[thm]{Lemma}
\newtheorem{cor}[thm]{Corollary}
\newtheorem{prop}[thm]{Proposition}

\newtheorem{fact}[thm]{Fact}

\theoremstyle{definition}
\newtheorem{definition}[thm]{Definition}

\theoremstyle{remark}
\newtheorem{remark}[thm]{Remark}

\theoremstyle{remark}
\newtheorem{example}[thm]{Example}

\theoremstyle{remark}
\newtheorem{claim}[thm]{Claim}

\theoremstyle{remark}

\theoremstyle{remark}

\DeclareMathOperator{\inte}{int}

\begin{document}
\bibliographystyle{plain}

\title{Tame topology over definable uniform structures}

\author{Alfred Dolich and John Goodrick}

\begin{abstract}
A \emph{visceral structure} on $\mathfrak{M}$ is given by a definable base for a uniform topology on its universe $M$ in which all basic open sets are infinite and any infinite definable subset $X \subseteq M$ has non-empty interior. 

Assuming only viscerality, we show that the definable sets in $\mathfrak{M}$ satisfy some desirable topological tameness conditions. For example, any definable function $f: M \rightarrow M$ has a finite set of discontinuities; any definable function $f: M^n \rightarrow M^m$ is continuous on a nonnempty open set; and assuming definable finite choice, we obtain a cell decomposition result for definable sets. Under an additional topological assumption (``no space-filling functions''), we prove that the natural notion of topological dimension is invariant under definable bijections. These results generalize some of the theorems proved by Simon and Walsberg in \cite{Simon_Walsberg}, who assumed dp-minimality in addition to viscerality. In the final section, we construct new examples of visceral structures.
\end{abstract}

\maketitle

\section{Introduction}

The present work contributes to the growing body of results in model theory about topological tameness properties of definable sets in various classes of structures. We consider prototypical examples to be \textbf{o-minimal structures}, such as the theory of real closed fields, and \textbf{P-minimal structures}, such as the $p$-adic field. In both of these cases, the classes of definable sets and functions satisfy many desirable topological properties: definable functions are not too far from being continuous; there is a natural topological dimension function which is invariant under definable bijections; and definable sets (even in Cartesian powers of the structure) have \emph{cell decompositions}, which are finite partitions into pieces which are ``topologically nice.'' See \cite{HM} for the case of P-minimal fields, and see \cite{vdd} for o-minimality.

In this article, we introduce a new common generalization of o-minimality and P-minimality which we call \emph{viscerality}. This may be the most general class of theories studied so far in which it is reasonable to hope to prove cell decomposition and near-continuity of definable functions. As we point out below, this context includes the dp-minimal definable uniform structures investigated by Simon and Walsberg \cite{Simon_Walsberg}, but also includes structures which are not even NIP.

For visceral theories, we establish the following facts in this paper:

\begin{enumerate}
\item All definable functions are continuous almost everywhere (Proposition~\ref{contvisc}); 

\item Under the hypothesis of Definable Finite Choice (that is, the existence of Skolem functions for finite sets), there is a cell decomposition theorem (Theorem~\ref{celldecomp}); and 

\item Under an additional topological hypothesis (the absence of ``space-filling functions''), the natural topological dimension function in visceral theories is invariant under definable bijections (Theorem~\ref{dim_inv}). 
 
 \end{enumerate}

In the final section of the paper, we construct new examples of visceral expansions of ordered Abelian groups including an example with the independence property.

\subsection{Detailed summary of results}

We recall that a uniform structure on $M$ is given by a family $\Omega \subseteq \mathcal{P}(M \times M)$ such that each $E \in \Omega$ defines ``uniform balls'' $E[a]$ ``centered'' at points $a \in M$, which satisfy certain axioms. (In Sections 2 and 3 below, we will give precise definitions of all relevant concepts.) This framework gives a simultaneous generalization of the interval topology on an ordered Abelian group and the usual topology on the $p$-adic field.

Given a uniform structure on the universe $M$ of a structure $\mathfrak{M}$ with a definable base, we say that $\mathfrak{M}$ is \emph{visceral} if every ball is infinite and every infinite definable subset of $M$ has interior (by which we mean has {\em non-empty} interior). We say that the theory $T$ is visceral if all of its $\omega$-saturated models are. All topological notions below refer to the topology in which a neighborhood basis for each point $x \in M$ is given by the set of $M$-definable balls $D[x]$ centered at $x$, and we generally assume that $\mathfrak{M}$ is visceral and sufficiently saturated.



In Section 2, we recall the precise definition of a uniform structure and set some notational conventions. In Section 3 we introduce the concept of a visceral first-order theory (Definition~\ref{visceral}) and prove a series of general results: all definable unary functions are continuous (according to the visceral definable uniform topology) off a finite set (Proposition~\ref{contvisc}); a finite union of definable sets with empty interior has empty interior (Proposition~\ref{open}); and under the extra assumption of Definable Finite Choice (DFC), a cell decomposition theorem is obtained (Theorem~\ref{celldecomp}). Note that DFC automatically holds in all ordered structures and in all P-minimal fields. Next we define a natural topological dimension function on definable sets and show that it is invariant under definable bijections, at least if we make the extra assumption of ``no definable space-filling functions'' (Theorem~\ref{dim_inv}). We could not see how to establish this property for a general visceral theory nor construct a visceral example with space-filling functions, but at least some of the most important classes of examples (those which are dp-minimal or which satisfy algebraic exchange) have no such functions.

In the final part of Section 3, we discus the special case of an ordered Abelian group in which the interval topology yields a visceral uniform structure. We call such groups \emph{viscerally ordered}, and they were our original motivation for studying the more general concept of visceral structures.

In Section 4  we construct some  examples of viscerally ordered expansions of divisible Abelian groups including examples with the independence property.
\subsection{Comparison with related work}

Simon and Walsberg \cite{Simon_Walsberg} recently proved some similar results for visceral dp-minimal theories (although they did not call them such; what we call viscerality, they called ``(\textbf{Inf})''). For instance, they also proved that definable functions are continuous almost everywhere and that the natural topological dimension function is invariant under definable bijections. We do not assume dp-minimality or even NIP, and in that sense our results are more general; on the other hand, we needed Definable Finite Choice for our cell decomposition theorem and a few other results, whereas Simon and Walsberg compensate for the lack of DFC by decomposing definable sets into graphs of ``continuous multi-valued functions.''

In William Johnson's Ph.D. thesis \cite{Johnson_thesis}, it is shown that any dp-minimal, not strongly minimal field has a definable uniform structure which is visceral in our sense, furnishing many interesting examples of visceral theories.

The cell decomposition theorem \emph{par excellence} in model theory is that for o-minimal structures by Knight, Pillay, and Steinhorn \cite{KPS}. The cell decomposition theorem we obtain for viscerally ordered Abelian groups is obviously much weaker than this classic result, since, for instance, a $1$-cell for us may contain infinitely many connected components.

It is worth clarifying what our results mean in the special case of P-minimal fields. In the literature, there are now various different results which are known as ``cell decomposition'' for the $p$-adic field or more generally for P-minimal fields, of which the most celebrated is Denef's cell decomposition for semi-algebraic sets \cite{Denef}. But for us, the most relevant is a recent variation by Cubides-Kovacsics, Darni\`ere and Leenknegt \cite{Cub}, wherein they establish a ``Topological Cell Decomposition'' for P-minimal fields. Our Theorem~\ref{celldecomp} applies to the P-minimal case (where Definable Finite Choice and the exchange property for algebraic closure always hold), but our conclusion is slightly weaker than that of \cite{Cub} since we do not establish that the cells are ``good'' (either relatively open or relatively interior-free in the set we are decomposing). Nonetheless, our cell decomposition is still strong enough to derive what they call the Small Boundaries Property (see Corollary~\ref{small_boundaries} below). 


Our notion of viscerality is very similar to what Mathews called a ``$t$-minimal topological structure'' ($t$ stands for ``topological''). We recall that a \emph{first-order topological structure} is a first-order structure $\mathcal{M}$ on which there is a definable family $\{\varphi(M; \overline{a}) \, : \, \overline{a} \in M^n\}$ which forms a basis for a topology on the universe $M$, and this structure is called \emph{$t$-minimal} if the topology induced by the $\varphi(M; \overline{a})$ satisfies the following three conditions:

\begin{enumerate}
\item It is $T_1$;
\item There are no isolated points in $M$; and
\item Every definable $X \subseteq M$ has only finitely many boundary points.
\end{enumerate}

For comparison, our definition of ``visceral definable uniform structure'' is equivalent to (2) and (3) above plus the added condition that the topology comes from a definable uniform structure, but with no requirement that the topology be $T_1$. In \cite{Mathews}, Mathews proves a cell decomposition result for a certain class of $t$-minimal structures using the same definition of ``cell'' as we use (that is, a definable set such that some coordinate projection yields a homeomorphism onto an open set), but only in the case where the structure satisfies various other properties beyond mere $t$-minimality (elimination of quantifiers, finite Skolem functions, and the ``Local Continuity Property''). Rather confusingly, there is a competing definition of ``$t$-minimality'' in the literature from an unpublished note of Schoutens \cite{t_min}, in which yet another cell decomposition result is proven; however, Schouten's notion of ``$t$-minimal'' is more restrictive and fails to include even many weakly o-minimal ordered structures. 

Many thanks are in order to the anonymous referee for their extraordinarily helpful comments regarding the original version of this paper. Thanks to their efforts, many of the arguments herein have been substantially clarified (and in some cases, corrected). Any mistakes that may remain are of course our own fault.

\section{Uniform Structures}

Here we review some basic definitions and results concerning uniform structures. We do not claim that anything here is new, and in fact all of this material can be found in the textbook \cite{James_uniform}, but we include it here to make our paper more self-contained since it seems not to be very widely known.

We use the following notation for $D, E \subseteq M \times M$: $$D^{-1} = \{(y,x) \, | \, (x,y) \in D\};$$ $$D \circ E = \{(x,y) \, | \, \exists z \left[ (z, y) \in D \textup{ and } (x,z) \in E \right] \}.$$ We use $D^2$ as shorthand for $D \circ D$.

\begin{definition}
\label{uniform_struct}
Given a set $M$, a \emph{uniform structure on $M$} is a collection $\Omega \subseteq \mathcal{P}(M \times M)$ such that;
\begin{enumerate}
\item $\Omega$ is a \emph{filter}: that is, if $D, E \in \Omega$ then $D \cap E \in \Omega$, and if $D \in \Omega$ and $D \subseteq E \subseteq M \times M$, then $E \in \Omega$;
\item $\Delta M \subseteq D$ for all $D \in \Omega$, where $\Delta M = \{(x,x) \, : \, x \in M\}$;
\item If $D \in \Omega$, then $D^{-1} \in \Omega$; and
\item If $D \in \Omega$, then there is some $E \in \Omega$ such that $E^2 \subseteq D$.
\end{enumerate}
\end{definition}

In the context of a uniform structure as above, the sets $D \in \Omega$ are often called \emph{entourages}.

\begin{definition}
\label{base}
A \emph{base for a uniform structure on $M$} is a collection $\mathcal{B} \subseteq \mathcal{P}(M \times M)$ such that:
\begin{enumerate}
\item $\mathcal{B}$ is a \emph{base for a filter}: that is, $\mathcal{B} \neq \emptyset$ and if $D_1, D_2 \in \mathcal{B}$ then there is some $E \in \mathcal{B}$ such that $E \subseteq D_1 \cap D_2$;
\item $\Delta M \subseteq D$ for all $D \in \mathcal{B}$;
\item If $D \in \mathcal{B}$, then $E \subseteq D^{-1}$ for some $E \in \mathcal{B}$; and
\item If $D \in \mathcal{B}$, then there is some $E \in \mathcal{B}$ such that $E^2 \subseteq D$.
\end{enumerate}
\end{definition}

Given a base $\mathcal{B}$ for a uniform structure on $M$, the uniform structure $\Omega$ generated by $\mathcal{B}$ is simply the filter on $M \times M$ generated by $\mathcal{B}$, that is, the collection of all $E \subseteq M \times M$ such that there is some $D \in \mathcal{B}$ such that $D \subseteq E$. Conversely, to describe a uniform structure $\Omega$ on $M$, it is sufficient to give a subcollection $\mathcal{B} \subseteq \Omega$ which is a base and which generates $\Omega$.

For examples of uniform structures, suppose that $M$ is endowed with a pseudometric $\rho \, : \, M \rightarrow [0, \infty)$ (i.e. $\rho$ is symmetric, vanishes on $\Delta M$, and satisfies the triangle inequality, but $\rho(x,y) = 0$ does not necessarily imply that $x = y$), in which case there is a corresponding uniform structure on $M$ which is generated by the base $\{D_\varepsilon \, : \, \varepsilon \in (0, \infty) \}$, where $$D_\varepsilon := \{(x,y) \in M^2 \, | \, \rho(x,y) < \varepsilon\}.$$ In fact, as observed in chapter 1 of \cite{James_uniform}, a partial converse is true: any uniform structure on $M$ with a countable base arises from some pseudometric on $M$ via the construction above.\footnote{Thus uniform structures may seem like only a mild generalization of pseudomentrics, but for the present work we prefer the point of view that comes from thinking directly in terms of entourages without being encumbered by having to deal with numerical values of ``$\varepsilon$''.}

Given a uniform structure $\Omega$ on $M$, $D \in \Omega$, and $x \in M$, the set $$D[x] := \{y \in M : (x,y) \in D\}$$ is called a \emph{ball (centered at $x$)}.  

\begin{definition}
\label{unif_top}
If $\Omega$ is a uniform structure on $M$, {\em the uniform topology on $M$ induced by $\Omega$} is the topology such that $U \subseteq M$ is open if for every $x \in U$ there is $D \in \Omega$ so that $D[x] \subseteq U$.  
\end{definition}

The fact that the construction above yields a topology ($\emptyset$ and $M$ are open; the intersection of two open sets, and the union of arbitrary collections of open sets, are open) is left as a straightforward exercise to the reader. Note that the sets $D[x]$ are not necessarily open in this topology, but rather form a base for the neighborhoods of $x$ (fixing an $x \in M$, varying $D \in \Omega$).

An alternative way to describe the uniform topology on $M$ induced by $\Omega$ is that the topological closure of $X \subseteq M$ is $$\overline{X} = \bigcap_{D \in \Omega} \bigcup_{a \in X} D[a].$$ The equivalence of this with Definition~\ref{unif_top} follows from Proposition 3.6 of \cite{James_uniform}.

Now we discuss topological separation properties of  uniform topologies.

\begin{definition}
A uniform structure $\Omega$ on $M$ is \emph{separated} if $$\bigcap_{D \in \Omega} D = M \times M.$$
\end{definition}

A simple way to construct non-separated uniform structures on $M$ is to note that for any equivalence relation $E$ on $M$, the set $\{E\}$ is the base for a uniform structure on $M$, and unless $E$ is the equality relation, this structure will not be separated. Also note that when $E$ is not equality, the corresponding uniform topology is not $T_0$.

\begin{prop}
\label{separation_axioms}
If $\mathcal{T}$ is the uniform topology generated by a uniform structure $\Omega$ on $M$, then the following conditions are all equivalent:
\begin{enumerate}
\item $\Omega$ is separated.
\item $\mathcal{T}$ is $T_0$: for any two distinct points $x, y \in M$, the collection of neighborhoods of $x$ is not equal to the collection of neighborhoods of $y$.
\item $\mathcal{T}$ is $T_1$: for any point $x \in M$, the set $\{x\}$ is closed.
\item $\mathcal{T}$ is $T_2$ (Hausdorff): for any two distinct points $x, y \in M$, there are neighborhoods $U$ of $x$ and $V$ of $y$ which do not intersect.
\end{enumerate}
\end{prop}

\begin{proof}
That (1) is equivalent to (3) is Proposition 3.5 of \cite{James_uniform}, and that (1) is equivalent to (4) follows from the discussion following Proposition~3.7 of \cite{James_uniform}. Therefore the only nontrivial implication left to check is that (2) implies (3). Suppose that $\mathcal{T}$ is not $T_1$, so there is some point $x \in M$ such that $\{x\}$ is not closed. This means that there is some $y \in M$ with $y \neq x$ such that for every $D \in \Omega$, the neighborhood $D[x]$ of $x$ contains $y$. To show that $\mathcal{T}$ is not $T_0$, it will be sufficient to prove that these points $x$ and $y$ have the same neighborhoods, or that any ball centered at $x$ contains a sub-ball centered at $y$, and conversely any ball centered at $y$ contains a sub-ball centered at $x$.

So suppose that $D[x]$ is any ball centered at $x$. By Definition~\ref{uniform_struct} there is some $E \in \Omega$ such that $E^2 \subseteq D$. By our assumption that $y$ lies in all neighborhoods of $x$, we have that $y \in E[x]$.  Hence if $z$ is any element of $E[y]$, we have that $(x,y) \in E$ and $(y,z) \in E$, thus $(x,z) \in D$, or in other words, $z \in D[x]$; therefore $E[y] \subseteq D[x]$, and $D[x]$ is a neighborhood of $y$. Conversely, suppose that $D[y]$ is any ball centered at $y$. By Definition~\ref{uniform_struct}, $D^{-1} \in \Omega$, and also there is some $E \in \Omega$ such that $E^2 \subseteq D^{-1}$. Then if $z$ is any element of the ball $E^{-1}[x]$, we have $(z,x) \in E$, and also $(x,y) \in E$ (again using our assumption on the points $x$ and $y$), so that $(z,y) \in E^2 \subseteq D^{-1}$, and hence $z \in D[y]$, as desired.

\end{proof}

If $\Omega$ is a uniform structure on $M$, then $M^n$ has the usual product topology, and we will often refer to topological properties of subsets $X$ of $M^n$ accordingly. We will also refer to \emph{balls} $B \subseteq M^n$, which are simply products $B_1 \times \ldots \times B_n$ of balls $B_i = D_i[x_i]$ as defined in the previous paragraph.

\section{Cell Decomposition and Dimension in Visceral Theories}

Now we come to the main definitions of the paper. Throughout, ``definable'' means $A$-definable for some set of parameters $A$.

We note that much of the work in this section has parallels in \cite{MMS} and \cite{Mathews} though the context of the current paper is different from that considered in those papers.

\begin{definition}
If $\mathfrak{M} = (M, \ldots)$ is a structure, a \emph{definable uniform structure on $\mathfrak{M}$} is a base $\mathcal{B}$ for a uniform structure on $M$ which is uniformly definable: that is, there are formulas $\varphi(x, y; \overline{z})$ and $\psi(\overline{z})$ (possibly over parameters from $M$) such that $$\mathcal{B} = \{\varphi(M^2; \overline{b}) \, : \, \mathfrak{M} \models \psi(\overline{b}) \}.$$ 
\end{definition}

\begin{remark}
\label{symmetry}

If $\mathcal{B}$ is a definable uniform structure, there is no harm in further assuming that every $D \in \mathcal{B}$ is \emph{symmetric} (that is, $D^{-1} = D$), since we can replace each $D$ in our base by $D \cap D^{-1}$ if necessary, and this will generate the same uniform structure on $\mathfrak{M}$. From now on, we always assume that definable uniform structures have this property.
\end{remark}

\begin{definition}
\label{visceral_def}
 We say that a definable uniform structure $\mathcal{B}$ on $\mathfrak{M}$ is \emph{visceral} if it satisfies the following two properties:
\begin{enumerate}
\item For every $x \in M$ and every $D \in \mathcal{B}$, the set $D[x]$ is infinite.
\item If $X \subseteq M$ is definable and infinite, then $X$ has nonempty interior in the uniform topology.
\end{enumerate}

\end{definition}

Note that condition (1) in the definition above is equivalent to saying that every ball in the uniform structure generated by $\mathcal{B}$ is infinite, since such balls are of the form $E[x]$ for which there is some $D \in \mathcal{B}$ with $D \subseteq E$. The second condition was called ``(\textbf{Inf})'' in the paper \cite{Simon_Walsberg}.

For our first simple consequence of viscerality, we recall that two points $x$ and $y$ in a topological space are called \emph{topologically indistinguishable} if the set of all neighborhoods of $x$ is equal to the set of all neighborhoods of $y$. Topological indistinguishability (via the uniform topology) yields an equivalence relation on visceral definable uniform structures whose classes we will denote by $[x]_\sim$.

\begin{lem}
\label{indist}
In a visceral definable uniform structure $\mathcal{B}$ on $\mathfrak{M}$, each class $[x]_\sim$ is finite.
\end{lem}

\begin{proof}
If some class $[x]_\sim$ were infinite, then $[x]_\sim \setminus \{x\}$ would be an infinite definable set without interior, contradicting the definition of viscerality.
\end{proof}

Although it seems that most ``natural'' examples of visceral definable uniform structures are $T_0$ (and hence Hausdorff, by Proposition~\ref{separation_axioms}), there are some which are not. We thank the anonymous referee for suggesting the following example.

\begin{example}
Let $M = \R \setminus \{0\}$ and let $\mathfrak{M} = (M, \prec, R)$ where $x \prec y$ is interpreted as $|x| < |y|$ and $R(x,y,z)$ is interpreted as $|x \cdot y| = |z|$. In this structure, there is a family of formulas $\varphi(x,y; a, b)$ such that whenever $a$ and $b$ are positive real numbers, $$\mathcal{M} \models \varphi(x,y; a, b) \Leftrightarrow a \cdot |x| < |y| < b \cdot |x|.$$ Let $\mathcal{B}$ be the family of all such $\varphi(x,y; a,b)$ with $0 < a < 1 < b$. It is a simple exercise to check that this family $\mathcal{B}$ satisfies the four conditions in the definition of being a base for a uniform structure in Definition~\ref{base}. (For example, the same uniform structure can be generated by the symmetric entourages $$D_a := \{(x,y) \in (\R \setminus \{0\})^2 \, : \, \frac{|x|}{a} < |y| < a \cdot |x| \}$$ as $a$ varies over all real numbers greater than $1$, and for condition (4), note that for any $a > 1$, we can pick $E = D_{\sqrt{a}}$ so that $E^2 = D_a$.)

In  the uniform topology, the numbers $x$ and $-x$ are topologically indistinguishable for any $x \in M$, so the topology is not $T_0$. In fact, this topology is homeomorphic to the infinite segment $(0, \infty)$ with the usual topology but with each point ``doubled.''

For the viscerality of $\mf{M}$, each ball $D[x]$ defines a union of two nonempty open intervals in $\R$, hence is infinite.  If $X \subseteq M$ is infinite and  $\mf{M}$-definable with parameters $\ob{a}=a_1 \dots a_n$ then since the relations $\prec$ and $R$ are definable in the o-minimal structure $(\R; <, 0, \cdot)$ there is an open interval $(c,d) \subseteq X$ so that 
$(c,d) \cap \ob{a}= \emptyset$ and either $c>0$ or $d<0$.   Consider the function $\sigma_{\ob{a}}: M \to M$ given by:
\[\begin{cases} \sigma_{\ob{a}}(x)=x & \text{if } x=\pm a_i \text{ for } 1 \leq i \leq n \\
\sigma_{\ob{a}}(x)=-x & \text{otherwise} \end{cases}.\]  Notice that $\sigma_{\ob{a}}$ is an automorphism of $\mf{M}$ fixing $\ob{a}$. But then $U=(c,d) \cup \sigma_{\ob{a}}[(c,d)] \subseteq X$.  As $U$ is open in the uniform topology, $X$ has non-empty interior and thus $\mf{M}$ is visceral.  
\end{example}

It will be convenient to reformulate condition (1) in the definition of viscerality so that it is clearly expressible by a singe first-order sentence. We thank the anonymous referee for suggesting the proof of the following Lemma (which is not immediately obvious if $\mathcal{B}$ is not separated):

\begin{lem} 
\label{visc_2}
 A definable uniform structure $\mathcal{B}$ on $\mathfrak{M}$ is visceral if and only if it satisfies the following two properties:
\begin{enumerate}
\item For any $D \in \mathcal{B}$ and any $x \in M$, there is an $E \in \mathcal{B}$ such that $E[x] \subsetneq D[x]$.
\item If $X \subseteq M$ is definable and infinite, then $X$ has nonempty interior in the uniform topology.
\end{enumerate}

\end{lem}

\begin{proof}
Note that condition (2) of the Lemma is identical to (2) of the definition of viscerality. On the one hand, if $\mathcal{B}$ satisfies condition (1) of the Lemma, then for any $D \in \mathcal{B}$ and $x \in M$, we can iteratively apply this condition to find $D, E_1, E_2, \ldots$ such that $D[x] \supsetneq E_1[x] \supsetneq E_2[x] \ldots$, and hence $D[x]$ is infinite.

On the other hand, suppose that $\mathcal{B}$ is visceral, and we will show that condition (1) of the Lemma holds. Suppose that $D \in \mathcal{B}$ and $x \in M$. By condition (1) from the definition of viscerality, $D[x]$ is infinite, and by Lemma~\ref{indist} there is some $y \in D[x]$ such that $y \nsim x$. This means that there is some $E \in \mathcal{B}$ such that either $y \notin E[x]$ or $x \notin E[y]$; but since $\mathcal{B}$ is assumed to be symmetric (see Remark~\ref{symmetry}), in fact $y \notin E[x]$. Now we can find some $E_1 \in \mathcal{B}$ such that $E_1 \subseteq E \cap D$, and $E_1[x] \subsetneq D[x]$, as desired.
\end{proof}

We record another consequence of viscerality which will be useful later:

\begin{lem}
\label{nested_balls}
Given a ball $B$ in a visceral definable uniform structure and $x \in B$, there is some sub-ball $B' \subseteq B$ such that $x \notin B'$.
\end{lem}

\begin{proof}
Say $B = D[y]$ and fix $x \in B$. By viscerality, $B$ is infinite; and since $B$ is infinite and definable, its interior $B^\circ$ is also infinite. By Lemma~\ref{indist}, there is some $z \in B$ such that $z \notin [x]_\sim$ and $z \in B^\circ$. Choose some entourage $E \in \mathcal{B}$ such that $(x,z) \notin E$, and recall our assumption that every $E \in \mathcal{B}$ is symmetric. Since $z \in B^\circ$, there is some $E_0 \in \mathcal{B}$ such that $E_0(z) \subseteq B$. Now pick $E_1 \in \mathcal{B}$ such that $E_1 \subseteq E_0 \cap E$ and let $B' = E_1[z]$. On the one hand, $B' = E_1[z] \subseteq E_0[z] \subseteq B$, so $B'$ is a sub-ball of $B$; and on the other hand, $B' \subseteq E[z]$ and $(z,x) \notin E$, so $x \notin B'$.
\end{proof}

\begin{definition}
\label{visceral}
The complete theory $T$ is \emph{visceral} if there is an $\omega$-saturated model $\mathfrak{M} \models T$ such that $\mathfrak{M}$ admits a visceral definable uniform structure.
\end{definition}

\begin{lem}
\label{elem_ext}
If $\mathfrak{M}$ is $\omega$-saturated and admits a visceral definable uniform structure, then any $\mathfrak{M}' \succ \mathfrak{M}$ also has a visceral definable uniform structure given by the same formulas.
\end{lem}

\begin{proof}
All of the axioms for being the base of a uniform structure (see Definition~\ref{base}) are clearly first-order and hence are preserved by elementary extensions, as is clause (1) of the definition of Lemma~\ref{visc_2}. As for clause (2) of Lemma~\ref{visc_2}, if there were some infinite $\overline{a}$-definable subset $\theta(M^{\prime}; \overline{a})$ of $M^{\prime}$ without interior, then we could pick some $\overline{a}$ from $M$ with the same type as $\overline{a}$, and $\theta(M; \overline{a})$ would be an infinite definable subset of $M$ without interior.
\end{proof}

One might naturally wonder whether there could be a structure $\mathfrak{M}$ which admits a visceral definable uniform structure $\mathcal{B}$ but such that in an $\omega$-saturated extension $\mathfrak{M}'$ of $\mathfrak{M}$, the corresponding definable uniform structure $\mathcal{B}'$ is not visceral. We do not know the answer to this question. Macpherson, Marker, and Steinhorn constructed an example (in \cite{MMS}, section 2.5) of a structure  $\mathfrak{M}$ which is densely ordered and weakly o-minimal, but such that in an $\omega$-saturated elementary extension of $\mathfrak{M}$ one can define an infinite discrete subset of the universe. However, there is no group structure definable on the universe of $\mathfrak{M}$, so while $\mathfrak{M}$ is ``visceral'' with respect to the usual order topology (in the sense that every infinite definable subset of the universe has interior), this topology does not seem to be the uniform topology of any definable uniform structure on $\mathfrak{M}$.

\textbf{From now until the end of this section, we assume that $T$ is a visceral theory and we work within some fixed $\omega$-saturated model $\mathfrak{M} \models T$.} Note that \emph{any} $\omega$-saturated model $\mathfrak{M}$ will support a visceral definable uniform structure, and by the previous Lemma, there is no harm in assuming that $\mathfrak{M}$ is a universal ``monster model.''

\textbf{Here and below, we will fix some visceral definable uniform structure $\mathcal{B}$ on $\mathfrak{M}$,} and all topological concepts (``open,'' ``continuous,'' etc.) will refer to this uniform topology, or to the corresponding product topology on $M^n$. Of course there may be other definable uniform structures on $\mathcal{M}$ other than $\mathcal{B}$, and not all of these may be visceral (see Example~\ref{dp_2} below).

We begin by recalling a very basic fact, which was also proved in \cite{Simon_Walsberg}.

\begin{prop} Any visceral theory satisfies \textbf{uniform finiteness}: for any $n$-tuple $\overline{y}$ of variables and any formula $\theta(x; \overline{y})$ there is an $N \in \omega$ such that for every $\overline{b} \in M^n$, if $\theta(M; \overline{b})$ is finite, then $|\theta(M; \overline{b})| \leq N$.
\end{prop}

\begin{proof} 
Note that any finite set $X \subseteq M$ has no interior since every ball is infinite. So if uniform finiteness failed for $T$, then by compactness we would have an infinite definable discrete $X \subseteq M$, violating viscerality.
\end{proof}

Another easy observation is that visceral structures are $t$-minimal in the sense of Mathews \cite{Mathews}:

\begin{lem}
\label{finite_frontier}
If If $\mathfrak{M}$ admits a visceral definable uniform structure, then for any definable $X \subseteq M$, all but finitely many points of $X$ are in its interior.
\end{lem}

\begin{proof}
The set $X \setminus \textup{int}(X)$ is definable, so if it were infinite, it would contain a ball, which is absurd, since any ball $D[x]$ is a neighborhood of the point $x$.
\end{proof}

Now we will begin to do a finer analysis of definable sets and functions in a visceral theory. To begin with, definable unary functions are well behaved:

\begin{prop}\label{contvisc} If $f: M \to M$ is definable then $f$ is continuous at all but finitely points.
\end{prop}

We postpone the proof, first establishing a fundamental Lemma.  Also we note that a result similar to the preceding Proposition was established in the context of dp-minimal densely ordered Abelian groups in \cite{Goodrick_dpmin}.  

In the study of weakly o-minimal structures, Macpherson, Marker, and Steinhorn \cite{MMS} used imaginary sorts encoding Dedekind cuts, which they called ``definable sorts.'' We will need to generalize this to our context. In the definition below, the sets $Z_{\overline{c}}$ are somewhat analogous to initial segments of an ordered structure.

\begin{definition}
Recall that the base $\mathcal{B}$ is presented as $$\mathcal{B} = \{\varphi(M^2; \overline{b}) \, : \, \overline{b} \in Z \}$$ where $Z \subseteq M^k$ is definable. A \emph{definable sort} is a definable family $A = \{Z_{\overline{c}} \, : \, \overline{c} \in W\}$ (where $W \subseteq M^\ell$ is definable) such that each $Z_{\overline{c}}$ is a nonempty definable subset of $Z$ which is ``downward closed:'' that is, if $\overline{b}_1, \overline{b}_2 \in Z$, $\overline{b}_1 \in Z_{\overline{c}}$, and $$\varphi(M^2, \overline{b}_2) \subseteq \varphi(M^2, \overline{b}_1),$$ then $b_2 \in Z_{\overline{c}}$.
\end{definition}

By abuse of notation, we will not distinguish between a definable sort $A$ and the definable set $W \subseteq M^\ell$ as in the definition above, and a \emph{definable function} $f: M^n \rightarrow A$ is synonymous with a definable function $f: M^n \rightarrow W$ in the usual sense.

Now we have the following simple Lemma, which is like Lemma 3.10 from \cite{MMS}. 

\begin{lem}\label{boundf}  Suppose that $f: B \to A$ is definable where $B \subseteq M$ is a ball and $A$ is a definable sort. Then there is a ball $B' \subseteq B$ and some $E \in \mathcal{B}$ such that for every $x \in B'$, $E \in f(x)$.
\end{lem}

\begin{proof}   Pick pairwise distinct $\{ b_i : i \in \omega\}$ in $B$. By the fact that elements of $A$ are downward closed plus compactness, there is some $E \in \mathcal{B}$ so that $E \in f(b_i)$ for all $i \in \omega$.  Thus the set $\{x \in B: E \in f(x)\}$ is infinite, and hence has interior by viscerality, so it contains a sub-ball $B'$ of $B$ as desired.
\end{proof}

\textit{Proof of Proposition~\ref{contvisc}:}  Suppose for contradiction that $f$ is discontinuous at infinitely many points.  Hence by viscerality we may find a ball $B \subseteq M$ so that $f$ is discontinuous on each $x \in B$.

We begin by noting that there is $N \in \omega$ so that if $y \in f[B]$ then $f^{-1}(y) \cap B$ has size at most $N$.  Otherwise by compactness there is $y \in f[B]$ so that $f^{-1}(y) \cap B$ is infinite.  By viscerality there is a sub-ball $B'$ of $B$ so that $f(x)=y$ for all $x \in B'$.  But then $f$ is continuous on $B'$, violating our assumption on $B$.

Next we show that, without loss of generality, every point of the graph $\Gamma(f)$ of $f$ is an accumulation point of $\Gamma(f)$ in a strong sense (every neighborhood of every point of the graph contains infinitely many other points on the graph):

\begin{claim}
\label{omega_acc}
After replacing $B$ with some sub-ball if necessary, we may further assume that if $x \in B$ and $D, E \in \mathcal{B}$ then there are infinitely many $y \in D[x]$ such that that $f(y) \in E[f(x)]$.
\end{claim}

Suppose the Claim were false. Thus for any sub-ball $B'$ of $B$, there is some $x \in B'$ and $D, E \in \mathcal{B}$ such that $(D[x] \times E[f(x)]) \cap \Gamma(f)$ is finite. We claim that there is some $m \in \N$ such that for infinitely many $x \in B$, the set $(D[x] \times E[f(x)]) \cap \Gamma(f)$ has size at most $m$: for otherwise there would be only countably many points $x_0, x_1, \ldots$ in $B$ such that $(D[x] \times E[f(x)]) \cap \Gamma(f)$ is finite, and by applying Lemma~\ref{nested_balls} repeatedly, we could find a descending chain $B_0 \supseteq B_1 \supseteq \ldots$ of sub-balls of $B$ such that $x_0, \ldots, x_i \notin B_i$, and then by $\omega$-saturation we could obtain a sub-ball of $B$ which does not contain any point $x_i$, contradicting our hypothesis. Since the set $$\{x \in B \, : \, |(D[x] \times E[f(x)]) \cap \Gamma(f)| \leq m \}$$ is definable and infinite, by viscerality it contains some sub-ball $B'$ of $B$, and from now on we replace $B$ by this sub-ball $B'$.

By $\omega$-saturation of $\mathfrak{M}$, there is some $D^* \in \mathcal{B}$ such that for infinitely many $x \in B$, there exists some $E \in \mathcal{B}$ such that $|(D^*[x] \times E[f(x)]) \cap \Gamma(f)| \leq m$, and so by viserality this is true of every point $x$ in some sub-ball $B'$ of $B$, and again we replace $B$ by $B'$. By the same argument, we can also pick some $E^* \in \mathcal{B}$ such that, without loss of generality, for every $x \in B$, we have $|(D^*[x] \times E^*[f(x)]) \cap \Gamma(f)| \leq m.$

Now fix some $x \in B^\circ$. Pick some $D_0 \in \mathcal{B}$ such that $D_0[x] \subseteq B$ and also $D_0^2 \subseteq D^*$. We will show that the image $f(D_0[x])$ of the ball $D_0[x]$ under $f$ has no interior. Towards a contradiction, suppose that $E_0[y] \subseteq f(D_0[x])$, and without loss of generality $E_0 \subseteq E^*$. Choose $z \in D_0[x]$ such that $f(z) = y$. For any other $w \in D_0[x]$, since $(w,x)$ and $(x,z)$ are in $D_0$, we have that $w \in D^*[z]$; but as $z \in D_0[x] \subseteq B$, there are at most $m$ elements $w \in D^*[z]$ such that $f(w) \in E^*[f(z)]$, and in particular there are only finitely many $w \in D_0[x]$ such that $f(w) \in E_0[f(z)] = E_0[y]$. This is absurd, since by viserality $E_0[y]$ is infinite and was supposed to be contained in the image of $D_0[x]$.

Thus the function $f$ maps the infinite definable set $D_0[x]$ onto a set with no interior, so by viscerality the image $f(D_0[x])$ of $D_0[x]$ must be finite. But this contradicts our observation above that every fiber $f^{-1}(y)$ is finite, finishing the proof of Claim~\ref{omega_acc}.

For $x \in B$ let $g(x)$ be the set
\[\{E \in \mathcal{B} : \text{for all } D \in \mathcal{B} \text{ there is } y \in D[x] \text{ so that } f(y) \notin E[f(x)] \}.\]  Note that $g(x) \neq \emptyset$ on all of $B$ (as $f$ is discontinuous on all of $B$) and $g(x)$ is downward-closed, so $g$ is a definable function from $B$ into some definable sort.  

Now let $E \in \mathcal{B}$ and $B^{\prime}$ be any smaller ball contained in $B$. Take $E_0 \in \mathcal{B}$ such that $E_0^2 \subseteq E$, and pick any $x$ in the interior of $B'$, so say $D[x] \subseteq B'$ for $D \in \mathcal{B}$. By Claim~\ref{omega_acc}, there are infinitely many $y \in D[x]$ such that $f(y) \in E_0[f(x)]$, so by viscerality there is an even smaller ball $B''$ contained in $D[x]$ such that for any $y \in B''$, we have $f(y) \in E_0[f(x)]$. So for any $y_0, y_1 \in B''$, we have $(f(y_0), f(x)), (f(x), f(y_1)) \in E_0$, hence $(f(y_0), f(y_1)) \in E$; therefore for any $y_0 \in B''$, $E \notin g(y_0)$. But since $B'' \subseteq B'$ and $E$, $B'$ were chosen arbitrarily, this contradicts Lemma~\ref{boundf}. This finishes the proof of Proposition~\ref{contvisc}. \qed

\medskip

Next we work towards generalizing Proposition \ref{contvisc} to functions in an arbitrary number of variables.  To this end we need to establish a series of technical lemmas.  Our proof, in general outline, is similar to  proofs in Section 4 of \cite{MMS}, although the details are quite different.

\begin{lem}\label{qcont}  Suppose that $\mathfrak{M}$ is $\omega$-saturated and admits a visceral definable uniform structure. Then for every $n \in \N$, we have:

\begin{enumerate} \item[(I)]  Let $B \subseteq M^n$ be a ball (that is, a cartesian product $B_1 \times \ldots \times B_n$ of balls $B_i$) and $A$ a definable sort.  Suppose that $f: B \to A$ is definable. Then there is some $E \in \mathcal{B}$ so that $\{\ob{x} \in B : E \in f(\ob{x})\}$ has non-empty interior.

\item[(II)] Let $X \subseteq M^{n+1}$ be definable and let $\pi: M^{n+1} \to M^n$ be the projection onto the first $n$ coordinates.  Suppose that $\pi[X]$ has non-empty interior and there is $b \in M$ so that $b$ is in the interior of $X_{\ob{a}}$ (the fiber of $X$ above $\ob{a}$) for each $\ob{a} \in \pi[X]$. Then $X$ has non-empty interior. 

\end{enumerate}
\end{lem}

\begin{proof}  Let $(I)_n$ and $(II)_n$ be the claims of the Lemma specialized to a fixed value of $n \in \N$. We prove the lemma by induction on $n$ showing that the truth of $(I)_k$ and $(II)_k$ for all $k \in \{0, \ldots, n\}$ implies $(I)_{n+1}$, and that $(I)_n$ implies $(II)_n$. 

If $n=0$, both $(I)_0$ and $(II)_0$ are trivial.  Thus assume that $n=m+1$ and we have established $(I)_k$ and $(II)_k$ for every $k \leq m$.  We must first show that $(I)_n$ holds.  Fix $B=B_1 \times \dots \times B_{m+1}$ where the $B_i$ are balls and $f: B \to A$.  Without loss of generality we may assume that $\mathfrak{M}$ is very saturated.   Pick pairwise distinct $a^i_j \in B_i$ for $1 \leq i \leq m+1$ and $j \in \omega_i$, where $\omega_j$ denotes the $j$-th uncountable cardinal.  Also pick $E \in \mathcal{B}$ so that $E \in f(a^1_{j_1}, \dots, a^n_{j_n})$ for all $j_1 \dots j_n \in \omega_1 \times \dots \times \omega_n$.  We fix some notation for various definable sets.

Let $Z_{E}=\{\ob{x} \in B : E \in f(\ob{x})\}$.  We also define sets $Z_l \subseteq M^{m+1}$ for $1 \leq l \leq m+1$ recursively working backwards from $Z_{m+1}$.  Let $Z_{m+1}$ be the set\[\{ (y_1, \dots, y_{m+1}) \in B:  y_{m+1} \in \inte(Z_{E}(y_1, \dots, y_m, -))\},\] using the notation ``$Z_{E}(y_1, \dots, y_m, -)$'' to refer to the set of all $y$ such that $(y_1, \ldots, y_m, y) $ belongs to $Z_E$.  Given $Z_{l+1}$ let $Z_l$ be
\[\{(y_1, \dots, y_{m+1}) \in B : y_l  \in \inte(Z_{l+1}(y_1, \dots, y_{l-1},-,y_{l+1}, \dots, y_{m+1}))\}.\]  We claim that there are $j_1^*, \dots, j^*_{m+1} \in \omega_1 \times \dots \times \omega_{m+1}$ so that \[Z_l(a^1_{j_1}, \dots, a^{l-1}_{j_{l-1}}, a^l_{j^*_l}, \dots, a^{m+1}_{j^*_{m+1}})\] holds for all $j_1, \dots, j_{l-1} \in \omega_1 \times \dots \times \omega_{l-1}$ and all $l \in \{1, \ldots, m+1\}$.  We choose $j_1^*  \dots j_{m+1}^*$ recursively starting from $j_{m+1}^*$.  Note that for each 
$j_1, \dots j_m \in \omega_1 \times \dots \times \omega_m$ the set $\{x : (a^1_{j_1}, \dots a^m_{j_m},x) \in Z_{E}\}$ is infinite and by Lemma~\ref{finite_frontier} all but finitely many of its points lie in the interior; thus by cofinality considerations there must be $j_{m+1}^*$ as desired.  Given $j_{m+1}^*, \dots, j_{l+1}^*$ we note that for any fixed $j_1, \dots, j_{l-1} \in \omega_1 \times \dots \times \omega_{l-1}$ the set $\{x:Z_l(a^1_{j_1}, \dots, a^{l-1}_{j_{l-1}},x,j_{l+1}^*, \dots, j_{m+1}^*)\}$ is infinite and at most finitely many of its points are not in its interior, and once again by cofinality considerations we find $j^*_l$.

Now given $j_1^*, \dots, j_{m+1}^*$ we recursively construct open sets $U_1, \dots, U_{m+1}$ so that 
$U_l \subseteq B_1 \times \dots \times B_l$ and so that $Z_{l+1}(x_1, \dots, x_l, a^{l+1}_{j_{l+1}^*}, \dots a^{m+1}_{j_{m+1}^*})$ holds for all $(x_1, \dots, x_l) \in U_l$. For the case when $l = m+1$, we define $Z_{m+2}$ to be $Z_E$, and thus $U_{m+1}$ will be the set desired in order to establish $(I)_n$.  For $U_1$ note that as $Z_1(a^1_{j_1^*} \dots a^{m+1}_{j_{m+1}^*})$ holds there is an open neighborhood $U_1$ of $a^1_{j_1^*}$ so that $Z_2(x, a^2_{j_2^*}, \dots a^{m+1}_{j_{m+1}^*})$ for all $x \in U_1$.  Suppose we have constructed $U_l$.  Thus for each $(x_1, \dots, x_l) \in U_l$ we have that $Z_{l+1}(x_1, \dots, x_l, a^{l+1}_{j_{l+1}^*}, \dots, a^{m+1}_{j_{m+1}^*})$ holds.  As such $a^{l+1}_{j^*_{l+1}}$ lies in the interior of $Z_{l+2}(x_1, \dots, x_l, -, a^{l+2}_{j_{l+1}^*}, \dots, a^{m+1}_{j_{m+1}^*})$ for all $(x_1, \dots, x_l) \in U_l$.  Thus by $(II)_l$ there is $U_{l+1}$ as desired and thus establishing $(I)_{n}$.

Finally we show that $(I)_{n}$ implies $(II)_{n}$.  Thus suppose that $X \subseteq M^{n+1}$ and $b$ are as in the statement of $(II)_{m+1}$.  Without loss of generality we may assume that $\pi[X]$ is open.  For each $\ob{x} \in \pi[X]$ we let $$f(\ob{x})= \{E \in \mathcal{B} : E[b] \subseteq X_{\ob{x}} \}.$$  Thus we have $f: \pi[X] \to A$ for the associated definable sort $A$.  By $(I)_n$ there is $E \in \mathcal{B}$ and open $U \subseteq \pi[X]$ so that $E \in f(\ob{x})$ for all $\ob{x} \in U$.  Hence $U \times E[b] \subseteq X$ and we are done.
\end{proof}

The previous Lemma has the following useful consequence.

\begin{prop}\label{open}  If $X \subseteq M^n$ is definable and has non-empty interior and $X=X_1 \cup X_2$ with $X_1$ and $X_2$ both definable, then one of $X_1$ or $X_2$ has non-empty interior.
\end{prop}

\begin{proof}  We proceed by induction on $n$.  If $n=1$ the result is trivial by the viscerality assumption.  Hence assume that $n=m+1$. Without loss of generality $X$ is an open box of the form $V_1 \times \dots \times V_{m+1}$ in which the $V_i$ are nonempty definable open sets in the uniform topology and $X_1$ and $X_2$ are disjoint. Fix a sequence $\{b_i : i \in \omega\}$ of pairwise distinct elements of $V_{m+1}$.  We claim that for some $i < \omega$ there is an open set $U \subseteq V_1 \times \dots \times V_m$ and $j \in \{1,2\}$ so that $b_i$ is in the interior of $(X_j)_{\ob{a}}$ for all $\ob{a} \in U$.  So we suppose this fails, and show the following:

\begin{claim}
There is a nonempty open set $U_1 \subseteq V_1 \times \ldots \times V_m$ and $j_1 \in \{1,2\}$ such that $b_1 \in (X_{j_1})_{\ob{a}} \setminus \textup{int}((X_{j_1})_{\ob{a}})$ for all $\ob{a} \in U_1$.
\end{claim}

\begin{proof}
The set $Y := V_1 \times \ldots \times V_m$ can be definably decomposed into the following four pairwise disjoint subsets:

$$Y_1 := \{ \ob{a} \in Y \, : \, b_1 \in \textup{int}( (X_1)_{\ob{a}}) \},$$
$$Y_2 := \{ \ob{a} \in Y \, : \, b_1 \in \textup{int}( (X_2)_{\ob{a}}) \},$$
$$Y_3 := \{ \ob{a} \in Y \, : \, b_1 \in (X_1)_{\ob{a}} \setminus \textup{int} ( (X_1)_{\ob{a}} ) \},$$
and
$$Y_4 := \{ \ob{a} \in Y \, : \, b_1 \in (X_2)_{\ob{a}} \setminus \textup{int} ( (X_2)_{\ob{a}} ) \}.$$ 
By our assumption, neither of the two sets $Y_1$ nor $Y_2$ has interior, so by our induction hypothesis, at least one of the two sets $Y_3$ or $Y_4$ must have interior, and we can take $U$ to be the interior of either $Y_3$ or $Y_4$ and pick $j_1$ accordingly.

\end{proof}
  
 Repeating the argument used in the proof of the Claim above, we may construct an infinite chain $U_1 \supseteq U_2 \supseteq U_3 \supseteq \ldots$ of nonempty open sets $\{U_i : i \in \omega\}$ such that for each $i$ there is a $j_i \in \{1,2\}$ such that $b_i \in (X_{j_i})_{\ob{a}} \setminus \textup{int}((X_{j_i})_{\ob{a}})$ for all $\ob{a} \in U_i$. By compactness we may pick an $\ob{a} \in \bigcap_{i < \omega} U_i$. Then each $b_i$ is a non-interior point of either $(X_1)_{\ob{a}}$ or $(X_2)_{\ob{a}}$, so for some $j \in \{1,2\}$ the set $(X_j)_{\ob{a}} \setminus \textup{int}((X_j)_{\ob{a}})$ is infinite, contradicting Lemma~\ref{finite_frontier}.
 
 Hence there is an open set $U \subseteq V_1 \times \ldots \times V_m$, some $b_i \in V_{m+1}$, and $j \in \{1,2\}$ such that for every $\ob{a} \in U$, we have that $b_i \in \textup{int}((X_j)_{\ob{a}})$.  By Lemma \ref{qcont} $X_j$ has non-empty interior.
\end{proof}

Now we have our desired result on the continuity of functions in many variables.

\begin{thm}\label{gencont}  Suppose that $B \subseteq M^n$ is a ball and $f: B \to M$ is definable.  Then there is non-empty open definable $U \subseteq B$ such that $f$ is continuous on $U$.
\end{thm}

\begin{proof}  We prove the theorem by induction on $n$.  If $n=1$ the result follows from Proposition \ref{contvisc}.  Thus suppose we have the result for $m$ and we establish it for $n=m+1$.  Let $B=B_1 \times \dots \times B_{m+1}$.

Suppose the result fails.  By Proposition \ref{open} the set of all points at which $f$ is discontinuous must have interior and hence without loss of generality we assume that $f$ is discontinuous on all of $B$.  If $\ob{a} =  (a_1, \ldots, a_{m+1}) \in B$ and $D \in \mathcal{B}$, then we let $D[\ob{a}]$ denote the open box $D[a_1] \times \ldots \times D[a_{m+1}]$, and let \[g(\ob{a})= \{E \in \mathcal{B} : \text{ for all } D \in \mathcal{B} \text{ there is } \ob{b} \in D[\ob{a}] \text{ with } f(\ob{b}) \notin E[f(\ob{a})] \}.\]  Thus $g:B \to A$ (where $A$ is the appropriate sort determined by $g$), and by Lemma \ref{qcont} there is an open set $U \subseteq B$ and $E \in \mathcal{B}$ so that $E \in g(\ob{a})$ for all $\ob{a} \in U$. Once again without loss of generality we assume that $U$ is all of $B$.  By Proposition \ref{contvisc}, for each $\ob{a} \in B_1 \times \dots \times B_m$ the function $f(\ob{a},-)$ is continuous at all but finitely many points in $B_{m+1}$.  Arguing as in the proof of the previous Proposition, we may (after possibly shrinking $B$) find $b \in B_{m+1}$ so that $b$ lies in the interior of the continuity points of $f(\ob{a},-)$ for all $\ob{a} \in B_1 \times \dots \times B_m$. 

Pick $E' \in \mathcal{B}$ such that $E' \circ E' \subseteq E$, and for $\ob{a} \in B_1 \times \dots \times B_m$ let \[h(\ob{a})= \{D \in \mathcal{B} : 
\text{for all } y \in D[b], \,  f(\ob{a},y) \in E'[f(\ob{a},b)] \}.\]  Thus $h: B_1 \times \dots \times B_m \to C$ (where $C$ is the appropriate definable sort), and by Lemma \ref{qcont} there is $D \in \mathcal{B}$ and open $U \subseteq B_1 \times \dots \times B_m$ so that $D \in h(\ob{a})$ for all $\ob{a} \in U$.  Without loss of generality we assume that $U=B_1 \times \dots \times B_m$.  Also by induction we may find an open $U \subseteq B_1 \times \dots \times B_m$ so that $f(-,b)$ is continuous on $U$.  Without loss of generality we assume that $f(-,b)$ is continuous on all of $B_1 \times \dots \times B_m$.   Fix $(\ob{a},b) \in B$ and let $V \subseteq B_1 \times \dots \times B_m$ be an open neighborhood of $\ob{a}$ so that $f(\ob{x},b) \in E'[f(\ob{a},b)] $ for all $\ob{x} \in V$, which exists by continuity.  Now pick $(\ob{x},y) \in V \times D[b]$.  On the one hand, since $D \in h(\ob{x})$ and $y \in D[b]$, we have $$(f(\ob{x}, y), f(\ob{x},b)) \in E',$$ while on the other hand $$(f(\ob{x}, b), f(\ob{a},b)) \in E'$$ by the choice of $V$, and so since $E' \circ E' \subseteq E$ we conclude that $$(f(\ob{x},y), f(\ob{a},b)) \in E.$$ But the fact that this holds for any $(\ob{x},y)$ in a neighborhood around $(\ob{a},b)$ contradicts the fact that $E \in g(\ob{a},b)$.
\end{proof}

Our next goal is a general theorem showing that in a visceral theory definable sets may be partitioned into ``cells.''  Naturally our notion of cell will be quite weak.  In particular we must allow essentially arbitrary open sets as cells in that the assumption of viscerality places few restrictions on the definable open sets; this will also be made apparent in the examples constructed in the following section. We follow Mathews \cite{Mathews} in our definition of cell:

\begin{definition}  A definable set $X \subseteq M^n$ is a {\em cell} if for some coordinate projection $\pi: M^n \to M^m$ the set $\pi[X]$ is open and $\pi$ is a homeomorphism from $X$ to $\pi[X]$. In this case we also say that $X$ is \emph{an $m$-cell} (without claiming that this $m$ is necessarily unique).

We include the trivial case when $\pi: M^n \to M^n$ is the identity map, so any open definable $X \subseteq M^n$ is an $n$-cell.

\end{definition}

By convention we assume that  that $M^0$ is the one-point topological space, thus for any $\ob{a} \in M^n$ the singleton $\{\ob{a}\}$ counts as a cell. In the case where $n=1$, a definable set $X \subseteq M$ is a cell if either $X$ is open or $X$ is a single point.

First we assemble some basic observations about cells:

\begin{lem}\label{cellfacts}  
\begin{enumerate}
\item If $X \subseteq M^n$ is a cell and $f:X \to M$ is a continuous definable function then for any $1 \leq i \leq n$ the set $\Gamma(f,X,i):=$ 
\[\{(x_1, \dots, x_{i-1},y,x_i,\dots,x_n: (x_1, \dots, x_n) \in X \text{ and } f(x_1, \dots, x_n)=y\}\] is a cell.
\item If $X$ is a cell then $X$ has non-empty interior if and only if $X$ is open.
\end{enumerate}
\end{lem}

We recall a definition from \cite{Cub}:

\begin{definition}
\label{fsf} The theory $T$ has \emph{definable finite choice} (or DFC) if for every $\omega$-saturated model $\mathfrak{M}$ and every definable function $f: X \rightarrow M^n$ with domain $X \subseteq M^m$ such that for all $y \in f[X]$, $f^{-1}(y)$ is finite, there is a definable function $\sigma: f[X] \rightarrow X$ such that for every $x \in X$, $$(\sigma(f(x)), f(x)) \in \textup{Graph}(f).$$

\end{definition}

Note that any totally ordered structure has definable finite choice. It is also true, though less obvious, that the complete theory of the $p$-adic field $\Q_p$ has definable finite choice; see, for example, \cite{Cub}.

Now we prove our cell decomposition theorem:

\begin{thm} \label{celldecomp} Suppose that $T$ is visceral and has definable finite choice, $\mathfrak{M} \models T$ is $\omega$-saturated, $A \subseteq M$, and $n \in \N \setminus \{0\}$. Then:

$(I)_n$ For any $A$-definable $X \subseteq M^n$, there is a partition of $X$ into finitely many $A$-definable cells.

$(II)_n$ If $X \subseteq M^n$ and $f: X \to M$ is $A$-definable, then there is a partition of $X$ into $A$-definable cells $C_1 \dots C_m$ so that $f$ is continuous when restricted to $C_i$ for each $1 \leq i \leq m$.

\end{thm} 

\begin{proof}   

We will prove $(I)_n$ and $(II)_n$ by induction, showing that:

\begin{enumerate}
\item If $(I)_k$ and $(II)_k$ hold for all $k \in \{1, \ldots, n\}$, then $(I)_{n+1}$ holds; and

\item If $(I)_{n+1}$ and $(II)_n$ hold, then $(II)_{n+1}$ holds.
\end{enumerate}

For the base cases, notice that $(I)_1$ is trivial by viscerality and $(II)_1$ follows from Proposition \ref{contvisc}.

Now we assume that $(I)_k$ and $(II)_k$ hold for all $k \leq n$ and we prove $(I)_{n+1}$.  Suppose $X \subseteq M^{n+1}$ is $A$-definable.  As $\inte(X)$ is $A$-definable and a cell we may without loss of generality assume that $\inte(X)=\emptyset$.  Let $\pi: M^{n+1} \to M^n$ be projection onto the first $n$ coordinates.  By induction we may without loss of generality assume that $\pi[X]$ is a cell.  Suppose that $Y:=\pi[X]$ does not have interior.  Thus there is a coordinate projection $\pi_0: M^n \to M^m$ so that $\widetilde{\pi} := \pi_0 \upharpoonright Y$ is a homeomorphism of $Y$ onto $\pi_0[Y]$, which is open.  For convenience let us assume that $\pi_0$ is projection onto the first $m$ coordinates.  Now consider the set $Z=\{(x_1, \dots, x_m,y) : ( \widetilde{\pi}^{-1}(x_1, \dots, x_m),y) \in X\}$.  By induction we may partition $Z$ into $A$-definable cells $D_1, \dots, D_k$.  Setting \[D^*_i=\{( \widetilde{\pi}^{-1}(x_1, \dots, x_m),y): (x_1, \dots, x_m,y) \in D_i\},\] we easily check that the $D^*_i$ are $A$-definable cells partitioning $X$.  Hence we may assume that $Y$ is open.

For each $\ob{a} \in Y$ let $X_{\ob{a}}$ be the fiber of $X$ over $\ob{a}$.   By Lemma~\ref{finite_frontier}, there is an $N \in \omega$ so that $X_{\ob{a}}$ has at most $N$ non-interior points for all $\ob{a} \in Y$.   Since $T$ has definable finite choice, without loss of generality we reduce to the case that either $X_{\ob{a}}$ is a singleton for all $\ob{a} \in Y$ or that $X_{\ob{a}}$ is open for all $\ob{a}\in Y$.  First suppose that we are in the former case.  Thus $X$ is the graph of an $A$-definable function $f$ with domain $Y$, and by $(II)_n$ we may repartition $Y$ to reduce to the case that $f$ is continuous on $Y$, but then by Lemma \ref{cellfacts} $X$ is a cell.  Thus we may assume that for all $\ob{a} \in Y$ the fibre $X_{\ob{a}}$ is open.

Let $\pi_{n+1}: M^{n+1} \to M$ be the projection onto the last coordinate and let $W=\pi_{n+1}[X]$.  As $X$ has empty interior, by Lemma \ref{qcont} for  no $b \in W$ can $X_b$ have interior.  By induction $X_b$ has a partition into $Ab$-definable cells, none of which are open.  By compactness there are \[\psi_1(x_1, \dots, x_n,y), \dots, \psi_s(x_1, \dots, x_n,y)\] finitely many formulae with parameters from $A$ so that  for each $(\ob{a}, b) \in X$, for some $1 \leq i \leq s$ the set $\psi_i(x_1, \dots, x_n,b)$ defines a cell with empty interior and $\psi_i(\ob{a},b)$ holds.  Hence after partitioning $X$ we are reduced to considering the case where $X$ is a set defined by a single formula  $\psi(x_1, \dots x_n,y)$ as above. (It may not necessarily be the case that for this new $X$, the fibers $X_{\ob{a}}$ are open for every $\ob{a} \in Y$ as before, but this will not matter for the argument that follows.)  Furthermore after potentially partitioning $X$ again we may assume that there is a projection $\pi_l:\psi(M^n,b) \to M^l$ which is a homeomorphism onto its image for all $b \in W$.  For convenience assume that $\pi_l$ is projection onto the first $l$ coordinates.  Let $g(x_1, \dots, x_l,y)$ be the function producing the unique witness to $\exists x_{l+1}, \dots x_n\psi(a_1, \dots a_l, x_{l+1}, \dots, x_n,b)$ for each $b \in W$ and $a_1, \dots, a_l \in \pi_l[\psi(M^n,b)]$.  Thus the set $X$ is exactly the graph of the function $g$.  By induction we may partition the domain of $g$ into cells so that $g$ is continuous on each cell.  But then the graphs of $g$ restricted to each of these cells is a partition of $X$ into $A$-definable cells.  This establishes $(I)_{n+1}$.

Lastly we show that $(I)_{n+1}$ and $(II)_n$ imply $(II)_{n+1}$.  Let $X \subseteq M^{n+1}$ be $A$-definable and $g: X \to M$ be an $A$-definable function.  By $(I)_{n+1}$ we may assume that $X$ is a cell.  First suppose that $X$ has no interior.  Then for some projection function $\pi:M^{n+1} \to M^m$ we have that $\pi$ is a homeomorphism between $X$ and $\pi[X]$.  For convenience assume that $\pi$ is projection onto the first $m$ coordinates.  Thus we may consider $g \circ \pi^{-1} : \pi[X] \to M$ as an $A$-definable function in the obvious way.  By $(II)_n$ we may partition $\pi[X]$ into $A$-definable cells $D_1, \dots, D_k$ so that $g \circ \pi^{-1}$ is continuous on each $D_i$.  Finally let \[ D_i^* = \{(x_1, \dots, x_{n+1}) \in X : (x_1, \dots, x_m) \in D_i \}.\]  Thus $g$ is continuous on each $D_i^*$ and the $D_i^*$ partition $X$.  Finally, if necessary, apply $(I)_{n+1}$ to partition each of the $D_i^*$ into cells.  Thus we may assume that $X$ has interior.  In this case first consider $X_1$ the set of all points in $X$ at which $g$ is continuous.  Clearly any partition of $X_1$ into cells will suffice for $(II)_{n+1}$.  Hence we only need to consider $g$ restricted to $X \setminus X_1$ but by Theorem \ref{gencont} $X\setminus X_1$ has empty interior and we are done. 
\end{proof}

In the general case of a visceral theory which does not necessarily have definable finite choice, one might hope for an even more general form of cell decomposition of sets into the graphs of definable continuous ``finite-to-one correspondences.'' Simon and Walsberg \cite{Simon_Walsberg} achieved this for dp-minimal visceral theories. We did not pursue this in the present article since our original motivating examples all have DFC.

\subsection{Topological dimension}

Next we consider a natural topological dimension function for visceral theories. This definition is not new; it was called  ``topological dimension'' by Mathews \cite{Mathews} and ``na\"{i}ve topological dimension'' by Simon and Walsberg \cite{Simon_Walsberg}.

\begin{definition}  For any structure $\mathfrak{M}$ and for any $X \subseteq M^n$, the {\em dimension} of $X$, written $\dim(X)$, is the largest $m \leq n$ so that $\pi[X]$ has non-empty interior for some coordinate projection $\pi: M^n \to M^m$.
\end{definition}

We establish a basic fact about dimension in the visceral context.

\begin{prop}\label{union}  If $X \subseteq M^n$ is definable and $X=X_1 \cup \dots \cup X_r$ for definable sets $X_i$ then $\dim(X)=\max\{\dim(X_i): 1 \leq i \leq r\}$.
\end{prop}

\begin{proof}  It suffices to consider the case where $X=X_1 \cup X_2$.  Clearly $\dim(X) \geq \max\{\dim(X_1),\dim(X_2)\}$.  Now suppose that $\pi : M^n \to M^m$ is a coordinate projection so that $\pi[X]$ has non-empty interior.  As $\pi[X]=\pi[X_1] \cup \pi[X_2]$, Proposition \ref{open} implies that one of $\pi[X_1]$ or $\pi[X_2]$ has non-empty interior.  Hence $\dim(X) \leq \max\{\dim(X_1), \dim(X_2)\}$.

\end{proof}

Next we work towards establishing that, under an additional topological hypothesis, $\dim$ is invariant under definable bijections.  We need a basic lemma.

\begin{lem}\label{invariance} Suppose $T$ has definable finite choice and let $\mathfrak{M}\models T$.  Suppose that $X \subseteq M^m$ is definable and $Y \subseteq M^n$ is definable with non-empty interior and $m<n$.  Then there is no definable bijection $f: X \to Y$.  
\end{lem}

\begin{proof}    Let $f: X \to Y$ be a counterexample to the Lemma.  First we show that we may assume that $f$ is a homeomorphism.  Take a partition of $X$ into cells $C_1, \dots, C_k$ so that $f$ is continuous when restricted to each $C_i$.  By Proposition \ref{open} at least one $f[C_i]$ has non-empty interior.  Thus without loss of generality $f$ is continuous.  Argue similarly to assume that $f^{-1}$ is continuous as well.

So it suffices to prove that the following statement holds for every $m$:

\medskip

\emph{If $X \subseteq M^m$ is definable, $n > m$, and $Y \subseteq M^n$ is definable with non-empty interior, then there is no definable homeomorphism $f : X \rightarrow Y$.}

\medskip

We prove this by induction on $m$. If $m=0$ then $X$ is finite and $Y$ is infinite so the result is trivial.  Hence assume we have the result for any definable $X \subseteq M^{l}$ for each $l \leq m$ and assume that we now have $X \subseteq M^{m+1}$ and $f: X \to Y$ a homeomorphism where $Y \subseteq M^n$ has non-empty interior and $n>m+1$.  Let $B=B_1 \times \dots \times B_n$ be a ball so that $B \subseteq Y$.  Let $a \in B_n$ and set $Z=B_1 \times \dots \times B_{n-1}\times \{a\}$.  First of all notice that $f^{-1}[Z]$ cannot have interior in $M^{m+1}$ by the continuity of $f^{-1}$.  Partition $f^{-1}[Z]$ into cells $C_1, \dots, C_r$ and let $\pi:M^n \to M^{n-1}$ be projection onto the first $n-1$ coordinates.  Thus $\pi \circ f$ maps $f^{-1}[Z]$ bijectively onto $B_1 \times \dots \times B_{n-1}$.  By Proposition~\ref{open} again, one of the sets $(\pi \circ f) [C_i]$ must have non-empty interior, say $i=1$.   As $C_1$ does not have interior there is a coordinate projection $\pi_0 : M^{m + 1} \rightarrow M^{l}$ (for some $l \leq m$) such that $\pi_0$ maps $C_1$ homeomorphically onto its image.  But then $\pi \circ f \circ \pi_0^{-1}$ yields a definable bijection from a subset of $M^{l}$ to a subset with non-empty interior in $M^{n-1}$, which is impossible by the induction hypothesis, since $n-1 > m \geq l$.

\end{proof}

From the previous Lemma, we immediately have:

\begin{cor}
($T$ visceral and has DFC) If $n \neq m$, then there is no definable bijection between $M^n$ and $M^m$.
\end{cor}

Another consequence is:

\begin{cor} 
($T$ visceral and has DFC) If $X \subseteq M^n$ is definable and has non-empty interior and $f:X \to M^n$ is a definable injection then $f[X]$ has non-empty interior.
\end{cor}

\begin{proof}   Suppose the result fails witnessed by $X$ and $f$.  Argue as above to reduce to the case where $f[X]$ is a cell $C$. As $C$ has empty interior there is $\pi: M^n \to M^m$ a coordinate projection so that $\pi$ is a bijection of $C$ onto its image.  But then $f^{-1} \circ \pi^{-1}$ violates Lemma \ref{invariance}.
\end{proof}

Unfortunately under only the assumption of viscerality we have not been able to prove that dimension is preserved under bijections.  The difficulty lies in the fact that given a cell $C \subseteq M^n$ and $\pi: M^n \to M^l$ a projection so that $\pi$ maps $C$ homeomorphically to its image and so that $\pi[C]$ is open, we have not been able to show that $\dim(C)=l$ since a priori there may be another projection $\pi^{\prime}:M^n \to M^k$ with $k>l$ so that $\pi^{\prime}[C]$ has interior.  In order to achieve this we need an additional condition.

\begin{definition}  Given a structure $\mathfrak{M}$ a {\em space-filling function} is a function $f : X \to Y$ where $X \subseteq M^m$ and $Y \subseteq M^n$ so that $f$ is surjective, $Y$ has non-empty interior, and $m<n$.  We say that a theory $T$ has no space-filling functions if in no model of $T$ is there a definable space filling function.
\end{definition}

\begin{definition}
\label{EP} $T$ has the \emph{exchange property} if for all $a_1, \ldots, a_{n+1}, b \in M$, if $b \in \acl(a_1, \ldots, a_{n+1}) \setminus \acl(a_1, \ldots, a_n)$, then $a_{n+1} \in \acl(a_1, \ldots, a_n, b)$.
\end{definition}

Recall that all $o$-minimal theories have the exchange property, but not all dp-minimal ordered groups have this property: for example, complete theories of divisible Abelian ordered groups with contraction maps are weakly $o$-minimal (see \cite{Kuhl}), hence dp-minimal.

Our motivating examples have no space-filling functions:

\begin{prop} \label{no_space} Suppose that  $T$ visceral and either is dp-minimal or has the exchange property. Then $T$ has no space-filling functions.
\end{prop}

\begin{proof} First we show that having a space-filling function implies that $T$ is not dp-minimal. Let $f=(f_1, \dots f_n):X \to Y$ with $X \subseteq M^m$ and $Y \subseteq M^n$ be a space-filling function.  Let $B=B_1 \times \dots \times B_n$ be a box in $Y$, where each $B_i$ is a ball.  As balls are infinite, for each $1 \leq i \leq n$ we may pick pairwise distinct elements $b^j_i \in B_i$ for $j \in \omega$.  For $1 \leq i \leq n$ consider the family of definable sets $\Xi_i=\{x \in X \wedge f_i(x) \in b_i^j : j \in \omega\}$.  Notice that the $\Xi_i$ represent the rows of a randomness pattern with $n$ rows in $m$ free variables. Hence the dp-rank of $X$ is at least $n$. But by the additivity of dp-rank (see \cite{additivity_dp_rank}) and the fact that $M$ is dp-minimal, the dp-rank of $X \subseteq M^m$ cannot be greater than $m$, leading to a contradiction.

Now suppose that $T$ satisfies the exchange property but there is $f: X \to Y$ a space-filling function.  Suppose all this data is definable over a set $E$.  As $Y$ has non-empty interior we may find $(a_1, \dots a_n) \in Y$ such that the set $\{a_1, \ldots, a_n\}$ is algebraically independent over $E$. Let $b_1, \dots, b_m \in X$ so that $f(b_1, \dots, b_m)=(a_1, \dots, a_n)$.  Now notice that $\{a_1, \dots a_n, b_1, \dots, b_m\}$ is a set containing an $E$-independent set of size $n$ but it is contained in $\acl(E \cup \{b_1, \dots, b_m\})$, clearly contradicting that $T$ has the exchange property.
\end{proof}

As noted above if $T$ has no space-filling functions then we have the desired property for cells.

\begin{lem}  
\label{good_cells}
Suppose that $T$ has no space-filling functions, $C \subseteq M^n$ is a cell, and there is a coordinate projection $\pi:M^n \to M^l$ mapping $C$ homeomorphically onto an open set $\pi[C]$. Then $\dim(C)=l$.
\end{lem} 

\begin{proof}  Suppose otherwise.  Thus $\dim(C)=m>l$ and let $\pi^{\prime}:M^n \to M^m$ be a projection so that $\pi[C]$ has non-empty interior.  But then $\pi^{\prime} \circ \pi^{-1} : \pi[C] \to \pi^{\prime}[C]$ is a surjection violating that $T$ has no space-filling functions.
\end{proof}

We can now prove our desired result:

\begin{thm}  \label{dim_inv} Suppose that $T$ is visceral, has definable finite choice, and has no space-filling functions.  If there is a definable bijection $f: X \to Y$ between the definable sets $X$ and $Y$, then $\dim(X)=\dim(Y)$.
\end{thm}

\begin{proof}  Notice that it suffices to prove, by symmetry, that $\dim(X) \leq \dim(Y)$. For notation assume that $X \subseteq M^m$ and $Y \subseteq M^n$.    Partition $X$ into cells $C_1, \dots, C_s$.  By Proposition \ref{union} $\dim(X)=\dim(C_i)=k$ for some $i$, say $i=1$.  Let $\pi_0: M^m \to M^k$ be a coordinate projection so that $\pi_0$ maps $C_1$ homeomorphically onto its image and so that $\pi_0[C_1]$ is open, which exists by the previous Lemma.  Let $D_1 \dots D_t$ be a partition of $f[C_1]$ into cells.  By Proposition \ref{open} one of $\pi_0 \circ f^{-1}[D_i]$ must have interior, say $i=1$.  Fix $\pi_1 :M^n \to M^l$ a coordinate projection so that $\pi_1$ maps $D_1$ homeomorphically onto its image and so that $\pi_1[D_1]$ is open.  Now note that $\pi_1\circ f \circ \pi_0^{-1}$ is a bijection from a subset of $M^k$ with non-empty interior to an open subset of $M^l$.   By  Lemma \ref{invariance} and its corollary $k=l$.  Thus $\dim(X) \leq \dim(Y)$.
\end{proof}

An immediate consequence of Proposition~\ref{no_space} and Theorem~\ref{dim_inv} is that in dp-minimal visceral theories with definable finite choice, our topological dimension function is invariant under definable bijections. One of the principal results of Simon and Walsberg in \cite{Simon_Walsberg} was that this is true even if one removes the hypothesis of definable finite choice.

Finally, we observe that if the definable uniform topology is Hausdorff and $T$ has the exchange property, then we have a useful alternative definition of the dimension function. As observed by Cubides-Kovacsics \emph{et al.} \cite{Cub} in the case of $P$-minimal fields, an old argument of Mathews \cite{Mathews} gives us the following:

\begin{thm}
\label{D_dim} Suppose that the topology on $\mathfrak{M}$ is Hausdorff and that $T$ has DFC and the exchange property. Let $\ll$ be the strict order on nonempty definable sets of $M^n$ given by $$B \ll A \Leftrightarrow B \subseteq A \textup{ and } B \textup{ has no interior in } A,$$ and for $X \subseteq M^n$ define $D(X)$ to be the foundation rank according to this order: that is, $D(X) \geq 0$ if $X \neq \emptyset$, and $D(X) \geq m+1$ if and only if there is some $Y \ll X$ such that $D(Y) \geq m$.

Then for every definable $X \subseteq M^n$, $\dim(X) = D(X)$.
\end{thm}

\begin{proof}
See \cite{Cub}, Corollary~3.4.
\end{proof}

As a corollary, we derive the following, which was originally proved for the special case of $P$-minimal fields in \cite{Cub}:

\begin{cor}
\label{small_boundaries}
Suppose that $T$ is visceral, has DFC and the exchange property, and the uniform topology is Hausdorff. 

If $M \models T$ and $X \subseteq M^n$ is definable, then $\dim(\overline{X} \setminus X) < \dim(X)$.
\end{cor}

\begin{proof}
The same proof as in Theorem~3.5 of \cite{Cub} applies, using the previous Corollary and our Proposition~\ref{union} (which they call (HM$_1$)), but as the anonymous referee noted, we need to make a minor correction to the original argument of Mathews for Lemma 8.14 of \cite{Mathews}, as follows.

From \cite{Mathews}, Lemma 8.14 is an ``if and only if'' statement, and we need only correct the following implication  (according to the notation there):

\textit{$(*)$ If $\mathfrak{M}$ is a Hausdorff first-order topological structure which satisfies the Cell Decomposition Property and the exchange property for algebraic closure, and if $X \subseteq M^n$ is definable with $\dim(X) \geq k+1$, then there is a definable $Y \subseteq X$ such that $\dim(Y) \geq k$ and $Y$ has no interior in $X$.}

In the statement above, note that:

$\bullet$ A ``first-order topological structure'' includes the context of definable uniform structures, the topology being the usual uniform topology;

$\bullet$ Mathews's Cell Decomposition Property holds in our context, because it is the same as our cell decomposition result (Theorem~\ref{celldecomp} above);

$\bullet$ The quantity ``$\dim(X)$'' here is defined by Mathews using cell decomposition, this being the maximal $n$ such that there exists some $n$-cell $C$ which forms part of a cell decomposition of the set $X$; and

$\bullet$ It turns out that $\dim(X)$ as so defined by Mathews will be the same as the topological definition of $\dim(X)$ which we use, due to Lemma 8.12 of \cite{Mathews}.

Now we recall the original proof of $(*)$ given by Mathews. By the definition of topological dimension, there is some coordinate projection map $\pi: M^n \rightarrow M^{k+1}$ such that the image $\pi[X]$ has interior. Define an equivalence relation $\overline{a} \equiv \overline{b}$ on $\pi[X]$ via equality of the first coordinates. Pick any $\overline{a} = (a_1, \ldots, a_{k+1})$ in the interior of $\pi[X]$ and let $Y = \pi^{-1}[[\overline{a}]_\equiv]$, the inverse image of the equivalence class of $\overline{a}$. Since $\pi[X]$ is open, $\dim(Y) \geq k$. Mathews incorrectly asserts that $Y$ must have no interior in $X$, but as noted by the referee, one can construct counterexamples, such as $X = \{ (x, y, z) \in \R^3 \, : \, xyz=0\}$ in the real-closed field $\R$, $\pi: \R^3 \to \R^2$ which maps $(x,y,z)$ onto $(x,y)$, and then the inverse image of the $\equiv$-class of $(0,0)$ in fact has interior in $X$.

To fix this argument, we just need to reduce to the case when $X$ is a cell. By Proposition~\ref{open} above, at least one cell $C$ in some cell decomposition of $X$ must satisfy $\dim(C) = \dim(X) \geq k+1$. Say that $C$ is an $\ell$-cell, so there is a coordinate projection $\pi : M^n \rightarrow M^\ell$ such that $\pi$ maps $C$ homeomorphically onto the open set $\pi[C]$. Since we are assuming the exchange property, by Proposition~\ref{no_space} there are no space-filling functions, and so by Lemma~\ref{good_cells} we have that $\ell = \dim(C) \geq k+1$.  Now proceed as before using $C$ in place of $X$, defining the equivalence relation $\equiv$ on $\pi[C]$ according to equality of first coordinates and letting $Y \subseteq C \subseteq X$ be the inverse image under $\pi$ of any element $\overline{a} \in \pi[C]$, and it is routine to check that (i) $\dim(Y) \geq \dim(X) - 1 \geq k$, and (ii) $Y$ has no interior in $C$, and hence not in $X$ either, as desired.
\end{proof}

\subsection{Viscerally ordered Abelian groups}

Our original motivation for investigating the concept of viscerality was the realization that, for divisible ordered Abelian groups, it provides a context in which well-known results for o-minimal structures can be generalized. In this subsection, we will clarify what viscerality means for such groups.

Throughout this subsection, let $\mathfrak{R}= ( R, +, <, \dots )$ be an expansion of an ordered Abelian group. Applying our notion of viscerality to the order topology gives the following definition.

\begin{definition}
The structure $\mathfrak{R}$ is \emph{viscerally ordered} if:

\begin{enumerate}
\item The ordering on $\mathfrak{R}$ is dense, and
\item Every infinite definable subset $X \subseteq R$ has interior (in the order topology).
\end{enumerate}

The complete theory $T$ is \emph{viscerally ordered} if all of its models are.

\end{definition}

It is possible for a theory of a densely ordered group to be visceral according to some definable uniform structure which does \textbf{not} generate the order topology, yet not be viscerally ordered, as the following example shows.

\begin{example}
\label{dp_2}
Consider $\mathfrak{R} = ( \R, +, <, \Q )$ (the ordered group of the reals under addition, expanded by a unary predicate for $\Q$) and let $T = \Th(\mathfrak{R})$. This structure was studied in \cite{DMS} and \cite{DG}, in which it was proven that $T$ has quantifier elimination, o-minimal open core, and dp-rank $2$.

The dense, codense definable set $\Q$ means that $\mathcal{R}$ is not viscerally ordered. However, if we consider the uniform structure generated by $$D_\epsilon = \{(x,y) \in \R^2 \, : \, |x - y| < \epsilon \textup{ and } x-y \in \Q\}$$ (where $\epsilon$ ranges over positive elements of $\R$), then by quantifier elimination it is clear that this generates a visceral definable uniform structure on $\mathfrak{R}$, and so $T$ is visceral.
\end{example}

The complete theory of any ordered structure has definable finite choice, so in particular our cell decomposition result applies to viscerally ordered Abelian groups.

If $\mathfrak{R}$ is a \textbf{divisible} ordered Abelian group, then we can summarize the relationship between various tameness notions for $T = \Th(\mathfrak{R})$ as follows:

$$ \textup{ (weakly) o-minimal } \Rightarrow \textup{ dp-minimal } \Rightarrow \textup{ viscerally ordered}$$

The first implication was shown in \cite{DGL} and the second implication was proved by Simon \cite{S}. In Section 4 below, we construct examples showing that the second implication cannot be reversed; indeed, we show how to build viscerally ordered divisible Abelian groups which are not even NIP. We conjecture that the first implication cannot be reversed, either.



If $\mathfrak{R}$ is a \textbf{densely ordered} Abelian group which is not necessarily divisible, then Simon's theorem does not apply: $\mathfrak{R}$ may be dp-minimal but not viscerally ordered. For example, consider $\mathfrak{R} = (\Z_{(p)}, +, <)$, where $\Z_{(p)} \subseteq \Q$ consists of all fractions $r/s$ whose denominator is relatively prime to a fixed prime number $p$; as shown in \cite{Goodrick_dpmin}, this theory is dp-minimal, but it is not viscerally ordered, since the set of all $p$-divisible elements is dense and codense.

There do exist densely-ordered, non-divisible groups which are viscerally ordered:

\begin{example}
Let $R = \Z \times \Q$ with the lexicographic ordering $<$ in which the $\Z$-coordinate dominates: $(a,b) < (c,d)$ if $a < c$, or else $a=c$ and $b < d$.

We will use a quantifier elimination result from \cite{JSW} which is a simplification of the more general quantifier elimination proved by Cluckers and Halupczok \cite{CH} for general ordered Abelian groups. First, note that the group $R$ is what is called \emph{non-singular} in \cite{JSW}: for every prime $p$, the quotient $R / pR$ is finite. For non-singular ordered Abelian groups, it is shown in \cite{JSW} that one has quantifier elimination in the language $\mathcal{L}$ containing the following symbols:

\begin{enumerate}
\item Symbols for $+$, $-$ (a unary function), and $\leq$;
\item For each natural number $n$ and each class $\overline{a}$ in $R / nR$, a unary predicate $U_{n,\overline{a}}$ for the preimage of $\overline{a}$;
\item Constant symbols for each point in the countable model $R$; and
\item For each prime $p$ and each $a \in R$ which is not $p$-divisble, a unary symbol for $H_{a,p}$, the largest convex subgroup of $R$ such that $a \notin H_{a,p} + p R$.
\end{enumerate}

In the structure $R = \Z \times \Q$, it is easy to check that the subgroups $H_{a,p}$ can only be $\{0\} \times \Q$ (if $a = (k,x)$ with $k$ not $p$-divisible) or $\{(0,0)\}$ (if $a = (0, x)$) and that the unary predicates $U_{n, \overline{a}}$ define open convex sets. From this it is clear that any definable set $X \subseteq R$ is either finite or has interior, hence $T = \Th(\langle R, <, + \rangle)$ is viscerally ordered. The theory $T$ is also dp-minimal by Proposition~5.1 of \cite{JSW} since $R$ is non-singular.

\end{example}

On the other hand, viscerally ordered Abelian groups are not too far from being divisible:

\begin{lem}
\label{dense_visceral}
If $\mathfrak{R}$ is viscerally ordered, then for any positive integer $n$ and any positive $\varepsilon \in R$, there is a $\delta \in R$ such that $0 < \delta \leq \varepsilon$ and $(0, \delta) \subseteq n R$.
\end{lem}

\begin{proof}
The definable set $nR$ is infinite, hence has interior since it is viscerally ordered. Therefore we may find elements $a, \delta \in R$ such that $0 < \delta \leq \varepsilon$ and the interval $(a, a + \delta) \subseteq nR$. Since any element of $(0, \delta)$ is the difference of two points from $(a, a+\delta)$ and $nR$ is a subgroup, $(0, \delta) \subseteq nR$.
\end{proof}

Notice that we do not claim any type of ``monotonicity theorem'' for a general viscerally ordered Abelian group.  In particular, we would like to be able to show that if $T$ is visceral, $\mathfrak{M} \models T$, and $f: M \to M$ is definable there is a cofinite open set $U \subseteq M$ so that if $x \in U$ then there is a neighborhood $V$ of $x$ so that $f$ is either monotone increasing, monotone decreasing, or constant when restricted to $V$.  We do not know if this holds in general, in fact even if $T$ is dp-minimal we do not know whether or not this holds.  We can  verify this in one situation.  To this end recall:

\begin{definition}  $T$ is called {\em locally o-minimal} if for any model $\mathfrak{M} \models T$, any definable $X \subseteq M$, and any $x \in M$ there is $\varepsilon>0$ so that $[x, x+\varepsilon) \cap X$ is either empty, $[x, x+\varepsilon), (x,x+\epsilon)$, or just ${x}$, and the same condition for $(x-\varepsilon, x]$.
\end{definition}

See \cite{TV} for generalities on local o-minimality.  In particular recall that any weakly o-minimal theory is locally o-minimal as is the theory of any ultraproduct of o-minimal theories.  We will construct examples of viscerally ordered locally o-minimal theories in the following section.  If we add the assumption of local o-minimality to viscerality, we achieve our desired monotonicity result:

\begin{prop}  If $T$ is viscerally ordered and locally o-minimal, $\mathfrak{M} \models T$ and $f: M \to M$ is definable then there is an open, definable, and cofinite set $U$ so that if $x \in U$ then there is a neighborhood $V$ of $x$ so that $f$ is either monotone increasing, monotone decreasing, or constant when restricted to $V$.
\end{prop}

\begin{proof}  This proposition follows \emph{mutatis mutandis} from the proof of Theorem 3.4 in \cite{MMS}.  All of the Lemmas 3.6, 3.7, 3.8, 3.9 of \cite{MMS} can be proved essentially the same way in the viscerally ordered and locally o-minimal context with only very minor changes.  Notice that Lemma 3.10 from \cite{MMS} is our Lemma \ref{boundf}.  Note that in particular the assumption of local o-minimality is exactly what is needed to guarantee that if $a \in M$ then for some interval I with left endpoint $a$ either $f(x)>f(a)$ for all $x \in I$, $f(x)<f(a)$ for all $x \in I$, or $f(x)=f(a)$ for all $x \in I$ while this conclusion apparently does not hold in the absence of local o-minimality.
\end{proof}

\section{Examples of Viscerally Ordered and Dp-minimal Theories}

We show how to construct examples of viscerally ordered theories.
 We begin with $(V,\Gamma)$ a valued ordered rational vector space.\footnote{We note that we choose to work with vectors spaces over $\Q$ for simplicity, but we could just as well consider ordered vector spaces over any ordered field $K$.}  Thus \[\langle V,+,<,0,\lambda\rangle_{\lambda \in \Q}\] is an ordered rational vector space and  $\langle \Gamma,<\rangle$ is a linear ordering with a largest element $\gamma^*$.   Further there is a map $v: V \to \Gamma$ so that:
\begin{itemize}
\item $v(\lambda x)=v(x)$ for all $\lambda \in \Q \setminus\{0\}$.
\item If $0 < x<y$ then $v(x) \geq v(y)$.
\item $v(x+y) \geq \min\{v(x),v(y)\}$.
\item $v(x+y)=\min\{v(x),v(y)\}$ if $v(x) \not= v(y)$.
\item $v(x)=\gamma^*$ if and only if $x=0$.
\item $v$ is onto $\Gamma$.
\end{itemize}

The axioms above imply that $v$ is a convex valuation, that is, \[\{b \in V : v(b)=c \text { and } b>0\}\] is convex for any fixed $c \in \Gamma \setminus \{\gamma^*\}$.

  Let $\mathcal{L}_{O}$ be a relational language in which $\langle \Gamma,< \rangle$ eliminates quantifiers.  Let $\mathcal{L}_{VS}=\{+,<,0,\lambda\}_{\lambda \in \Q}$ be the language of ordered rational vector spaces.  We consider $(V,\Gamma)$ as a structure in the language $\mathcal{L}_{\Gamma}=\mathcal{L}_{VS} \cup \mathcal{L}_{O} \cup \{v\}$.  Let $T_{\Gamma}$ be the $\mathcal{L}_{\Gamma}$ theory of $(V,\Gamma)$.

In order to establish quantifier elimination for $T_{\Gamma}$, we will use the following criterion:

\begin{fact}
\label{qe_crit}
Suppose that $T$ is a theory with the following property:

Whenever $\mf{B}_0$ and $\mf{B}_1$ are models of $T$, $\mf{A}$ is a common substructure of both $\mf{B}_0$ and $\mf{B}_1$, $A \neq B_0$, and $\mf{B}_1$ is $|A|^+$-saturated, then there is some $b_0 \in B_0 \setminus A$ and some $b_1 \in B_1$ such that $\qftp(b_0/A)=\qftp(b_1/A)$.

Then $T$ has quantifier elimination.
\end{fact}

\begin{proof}
This is Corollary B.11.10 of \cite{avv}.
\end{proof}

We now have our basic quantifier elimination result, which is inherent in \cite{Kuhlmann1997} but we sketch out a simple proof for completeness.

\begin{prop} The theory $T_{\Gamma}$ eliminates quantifiers in the language $\mathcal{L}_\Gamma$.
\end{prop}

\begin{proof}  We will apply the criterion in Fact~\ref{qe_crit} above. Let $\mf{B}_0=(V_0, \Gamma_0)$ and $\mf{B}_1=(V_1, \Gamma_1)$ be models of $T_{\Gamma}$.  Let $\mf{A}=(V', \Gamma')$ be a substructure of both $\mf{B}_0$ and $\mf{B}_1$ with $A \not= B_0$.   Furthermore assume that $\mf{B}_1$ is $|A|^+$-saturated. To establish quantifier elimination it suffices to find $b_0 \in B_0 \setminus A$ and $b_1 \in B_1$ so that 
$\qftp(b_0/A)=\qftp(b_1/A)$.

First suppose that there is $b_0 \in \Gamma_0 \setminus \Gamma'$.  Then we easily find $b_1 \in V_1$ as desired by quantifier elimination in the language $\mc{L}_{O}$.  Hence we may assume that $\Gamma_0=\Gamma'$.

Now let $b_0 \in B_0 \setminus A$.  We need to find a realization of $\qftp(b_0/A)$ in $B_1$.  By compactness and the saturation of $\mf{B}_1$ it suffices to realize any finite $\Delta(x) \subseteq \qftp(b_0/A)$.  Under our assumptions and after some simple rearrangements, $\Delta(x)$ may be assumed to be of the form
\[\{a_0 < x < a_1\} \cup \{v(x)=c\} \cup \{v(x-d_i)=c_i: 0 \leq i \leq n\},\]
where $a_0, a_1 \in V' \cup \{-\infty, \infty\}$; $c, c_i \in \Gamma'\setminus \{\gamma^*\}$; and $d_0 < d_1< \dots < d_n \in V'$.  Also as $a_0$, $a_1$ and all of the $d_j$ lie in $V'$ we may assume that $(a_0,a_1) \cap \{d_0, \dots, d_n\}=\emptyset$ (since otherwise if say $d_0 \in (a_0, a_1)$ then we can replace $a_0<x<a_1$ in $\Delta(x)$ by either $a_0 < x < d_0$, $d_0<x<a_1$, or $x=d_0$).

We begin by simplifying $\Delta(x)$.  We first claim that without loss of generality all $c_i \geq c$. Otherwise suppose that for example  $c_0<c$.  As $b_0$ realizes $\Delta(x)$ we have that $v(b_0)=c$ and $v(b_0-d_0)=c_0$.  Also 
$v(b_0-d_0) \geq \min\{c, v(d_0)\}$ and thus it must be the case that $v(d_0)<c$ and so \[v(b_0-d_0)=v(d_0)=c_0.\]
But then for any $b \in B_0$ if $v(b)=c$ it follows that $v(b-d_0)=c_0$.  As this will hold in any model of $T_{\Gamma}$ together with the open diagram of $\mf{A}$ the formula $v(x-d_0)=c_0$ is superfluous and can without loss of generality be eliminated from $\Delta(x)$.

Arguing similarly we may also assume that there is a single $c'\geq c$ so that $c_i=c'$ for all $0 \leq i \leq n$.

Next we may assume that $v(d_i-d_j)=c'$ for all $0 \leq i < j \leq n$.  First suppose that 
$v(d_i-d_j)<c'$.  But $v(b_0-d_j)=v(b_0-d_i+d_i-d_j)=v(d_i-d_j)<c'$ which is impossible.   Now suppose that $v(d_i-d_j)>c'$.
Suppose that $v(b-d_i)=c'$ for some $b \in B_0$.  Thus $v(b-d_j)=v(b-d_i+d_i-d_j)=v(b-d_i)=c'$.  Hence $v(x-d_i)=c'$ implies that $v(x-d_j)=c'$.  Again this will hold in any model of $T_{\Gamma}$ together with the open diagram of $\mf{A}$ and thus by eliminating formulas form $\Delta(x)$ we can assume without loss of generality that $v(d_i-d_j)=c'$ for all $0 \leq i < j \leq n$.  

Finally by similar simple arguments we can assume that $c'=c$.
Thus we need to show that:
\[\Delta(x)=\{a_0 < x < a_1 \} \cup \{v(x)=c\} \cup \{v(x-d_i)=c : 0 \leq i \leq n\}\] is realized in $\mf{B}_1$ given that it is realized in $\mf{B}_0$.

We will assume that $0 \leq a_0$ (the case that $a_1 \leq 0$ is identical). Notice that $v(d_i) \geq c$ for all $i$ since otherwise $\Delta(x)$ cannot be realized in $\mf{B}_0$.

We work in $\mf{B}_1$.  First suppose that $d_i<a_0<a_1<d_{i+1}$.  This implies that that $v(a_0) \geq v(d_i) \geq c$ and $v(a_1) \leq c$ and thus in turn $v(a_1)=v(d_{i+1})=c$.  Also $v(a_1-d_i) =c$, $v(d_{i+1}-a_0)=c$, $v(a_0-d_i) \geq c$, and $v(d_{i+1}-a_1) \geq c$.  We claim that $\frac{1}{2}(a_0+a_1)$ must realize $\Delta(x)$. 
 The fact that $v(\frac{1}{2}(a_0+a_1))=c$ is immediate.  Note that $v(\frac{1}{2}(a_0+a_1)-d_i)=v(a_0+a_1-2d_i)=v(a_0-d_i+a_1-d_i)$.  As $v(a_1-d_i)=c$, if $v(a_0-d_i)>c$ then $v(\frac{1}{2}(a_0+a_1)-d_i)=c$. Otherwise $v(a_0-d_i)=c$ and then also $v(\frac{1}{2}(a_0+a_1)-d_i)=c$ by convexity of $v$.  Checking that $v(d_{i+1}-\frac{1}{2}(a_0+a_1))=c$ is identical.  Also that $v(d_j-\frac{1}{2}(a_0+a_1))=c$ for $j \notin \{i, i+1\}$ follows easily.



Now assume that $d_n<a_0$.  (The case that $a_1<d_0$ is symmetric.)  It must be the case that $v(a_0-d_n) \geq c$ and that $v(a_1-d_n) \leq c$.  By the axioms for $T_{\Gamma}$ there must be $b_1 \in B_1 \cap (a_0, a_1)$ so that $v(b_1)=c$.  Any such $b_1$ realizes $\Delta(x)$. 
\end{proof}

Fix $(\Gamma, <)$ a linear order with largest element and let $\mathcal{X} \subseteq \bigcup_{n \in \omega}\mathscr{P}(\Gamma^n)$.  Let $\mathcal{L}_{OP}$ be a relational language with $\{<, P\}_{P \in \mathcal{X}} \subseteq \mathcal{L}_{OP}$ in which the theory of  the structure $\Gamma_{\mathcal{X}}=\langle \Gamma, <, P\rangle_{P \in \mathcal{X}}$ eliminates quantifiers.  We can now naturally expand the structure $(V,\Gamma)$ to a structure, $\mathcal{R}_{\mathcal{X}}$ in the language $\mathcal{L}_{\mathcal{X}}=\mathcal{L}_{OP} \cup \mathcal{L}_{VS} \cup \{v\}$.  Let $T_{\mathcal{X}}=\Th(\mathcal{R}_{\mathcal{X}})$.  As $\Gamma_{\mathcal{X}}$ eliminates quantifiers in $\mathcal{L}_{\mathcal{X}}$ arguing almost identically to the above proposition we have:

\begin{prop} The theory $T_{\mathcal{X}}$ eliminates quantifiers in the language $\mathcal{L}_{\mathcal{X}}$.
\end{prop}

Let $\mf{M}=(V^*, \Gamma^*)$  be a model of $ T_{\mc{X}}$.  Let $\mf{M}^1$ be the vector space sort, $V^*$,  with the induced structure from $\mf{M}$ (i.e. we add a predicate for every $\emptyset$-definable, in $\mf{M}$, set $X \subseteq (V^*)^n$).  Let $T^1_{\mc{X}} = \text{Th}(\mf{M}^1)$.
Using quantifier elimination for $T_{\mc{X}}$ we can easily show:

\begin{prop} $T_{\mathcal{X}}^1$ is viscerally ordered.
\end{prop}

In some senses the theory $T_{\mc{X}}^1$ is quite well-behaved.  We have:

\begin{prop} Let $\mf{M}=(V^*, \Gamma^*)$ be a model of  $T_{\mc{X}}$.  Suppose that $b \in V^*$, $A \subseteq V^*$ and $b \in \dcl(A)$.  Then $b$ is in the $\Q$-linear span of $A$.  In particular $T^1_{\mc{X}}$ satisfies exchange for definable closure.
\end{prop}

\begin{proof}   Let $\varphi(x,\ob{a})$ be a formula with $\ob{a} \subseteq A$, $x$ in the vector space sort and $\varphi(M, \ob{a})$ finite.  Applying quantifier elimination we may without loss of generality assume that $\varphi(x, \ob{a})$ is of the form $\alpha_1(x) \wedge \dots \wedge \alpha_n(x) \wedge \beta_1(x) \wedge \dots \wedge \beta_m(x)$ where the $
\alpha_i$ and $\beta_j$ are atomic or negated atomic formulae, the $\alpha_i$ are in the language of ordered rational vector spaces and the $\beta_i$ involve predicates from the language $\mc{L}_{OP}$.  By the o-minimality of rational ordered vector spaces $\alpha_1(M) \wedge \dots \wedge \alpha_n(M)$ consists of finitely many points and open intervals and all the isolated points and endpoints of the intervals must lie in the linear span of $A$.  Each $\beta_i(x)$ is of the form $A(v(t_1(x)), \dots, v(t_r(x)))$ where $A$ is a predicate or the negation of a predicate from $\mc{L}_{OP}$ and the $t_j$ are linear terms with parameters from $A$ (note some of these terms may not include the variable $x$).   It is straightforward to see that $\beta_i(M)$ is a disjoint union of clopen convex sets and finitely many points.  Futhermore these finitely many points are solutions to $t_j(x)=0$ for some $1 \leq j \leq r$ and thus lie in the $\Q$-linear span of $A$.  It follows that any element of $\varphi(M, \ob{a})$ must lie in the $\Q$-linear span of $A$.
\end{proof}

Thus via Proposition \ref{no_space} we have:

\begin{cor}  $T^1_{\mc{X}}$ has no space-filling functions.
\end{cor}

Recall that any dp-minimal theory extending that of divisible ordered Abelian groups is viscerally ordered \cite{S}. One can of course na\"{i}vely ask whether the converse holds.  We show that this is false in a very strong sense.

\begin{prop}\label{IP_example}  There is a viscerally ordered theory that interprets arithmetic.
\end{prop}

\begin{proof}  Let $\Gamma=(\omega+1, < )$ and let $\mc{X} =\bigcup_{n \in \omega}\mathscr{P}(\Gamma^n)$.  Then clearly 
models of $T^1_{\mc{X}}$ interpret models of  $\text{Th}(\langle \N, +, \cdot\rangle$).
\end{proof}

Next as noted earlier we show how to obtain viscerally ordered locally o-minimal theories.  For a linear order $\Gamma$ let $\Gamma_{<\gamma}=\{x \in \Gamma: x<\gamma\}$ where $\gamma \in \Gamma$.

\begin{prop}Let $(\Gamma,<, \gamma^*)$ be a dense linear order with no left endpoint and right endpoint $\gamma^*$. Let $\gamma \in \Gamma$.  If $\mathcal{X} \subseteq \bigcup_{n \in \N}\mathscr{P}((\Gamma_{<\gamma})^n)$, then $T^1_{\mathcal{X}}$ is locally o-minimal.
\end{prop}

\begin{proof} For convenience we may assume that $\Gamma_{<\gamma} \in \mc{X}$.  We need to find a reasonable language in which the structure $\Gamma_{\mathcal{X}}=(\Gamma,<,P)_{P \in \mathcal{X}}$ has quantifier elimination.  

\begin{claim}\label{boundqe}  There is a relational language $\mathcal{L}_{OP}$ in which the theory of $\Gamma_{\mathcal{X}}$ has quantifier elimination so that for all symbols $R \in \mathcal{L}_{OP}$ with $\{R\} \notin \{<, \gamma*$\} the interpretation of $R$ in $\Gamma$ is contained in $\bigcup_{n \in N}\mathscr{P}((\Gamma_{<\gamma})^n)$.
\end{claim}
\begin{proof}  We can consider $\Gamma_{<\gamma}$ as a structure in the language $\{<,P\}_{P \in \mathcal{X}}$ and let $\mathcal{L}_{OP}'$ be any larger relational language in which this structure has quantifier elimination.   $\Gamma_{\mathcal{X}}$ may be expanded into an $\mathcal{L}_{OP}=\mathcal{L}_{OP}' \cup \{\gamma^*\}$ structure by interpreting any new predicate, $R$, in $\mathcal{L}_{OP}$ exactly as it was interpreted in $\Gamma_{<\gamma}$.  It is easy to verify that if $\mathfrak{B}$ and $\mathfrak{C}$ are elementarily equivalent to $\Gamma_{\mathcal{X}}$ as $\mathcal{L}_{OP}$ structures with $\mathfrak{C}$ sufficiently saturated and $\mathfrak{A}$ is a substructure of both $\mathfrak{B}$ and $\mathfrak{C}$ then for any $b \in B$ we may find $c \in C$ so that $\qftp(bA)=\qftp(cA)$.  Hence we have quantifier elimination in the language $\mathcal{L}_{OP}$.
\end{proof}

We need to verify that $T^1_{\mathcal{X}}$ is locally o-minimal.  It suffices to show that if $\mf{M}=(V^*, \Gamma^*)$ is a model of $T_{\mc{X}}$ then any definable subset of $V^*$ meets the criterion for local o-minimailty.  As we have quantifier elimination in the language $\mathcal{L}_{\mathcal{X}}$, to check this it suffices to check that every set $X \subseteq V^*$ defined by an atomic formula with parameters satisfies the condition for local o-minimality.   By viscerality it suffices to check that no atomically defined set can accumulate to a point.  Suppose otherwise and $X$ is atomically definable and accumulates to $a \in V^*$.  After shifting, we may without loss of generality assume that $a=0$.  Also the only interesting case is when $X$ is defined by a formula of the 
form $R(v(t_1(x)), v(t_2(x)), \dots, v(t_n(x)), c_1, \dots, c_m)$ where $R$ is an $n+m$-ary relation symbol in $\mc{L}_{OP}$ different from $<$, the $t_i$ are terms from $\mc{L}_{VS}$, and the $c_j$ are arbitrary elements of $\Gamma^*$.  The terms $t_i(x)$ are of the form $\lambda_ix+d_i$ for $\lambda_i \in \Q$ and $d_i \in V^*$.  Note that if $d_i \not=0$ then the valuation of $t_i(x)$ is equal to that of $d_i$ for all sufficiently small $x$.   Thus at least one of the $d_i=0$ and we can ultimately reduce to considering only a predicate of the form $R(v(x), \dots, v(x), c_1, \dots, c_m)$.   But as $X$ accumulates to $0$ there must be $b \in X$ with arbitrarily large valuation.  Thus the interpretation of $R$ in $\Gamma^*$ must be unbounded but this contradicts our choice of $\mc{L}_{OP}$ in Claim \ref{boundqe}.

\end{proof}

\bibliography{modelth-vis}

\end{document}